\documentclass[12pt,reqno]{article}

\usepackage[X2,T1]{fontenc}
\usepackage{appendix}
\usepackage{tabularx}

\usepackage[a4paper,top=3cm,bottom=3cm,inner=3cm,outer=3cm]{geometry}
\usepackage{graphicx} 
\usepackage{amsmath, amssymb, amsfonts}
\usepackage{amsthm}
\usepackage{tikz-cd}
\usepackage{quiver}
\usepackage{enumerate}
\usepackage{eucal}
\usepackage{multicol}

\usepackage[backend=biber,style=alphabetic]{biblatex}
\usepackage{xcolor}
\definecolor{darkgreen}{rgb}{0,0.45,0}
\usepackage[
colorlinks, 
citecolor=red, 
urlcolor=darkgreen, 
final, 
hyperindex,  
linkcolor = black
]{hyperref}

\usepackage{amsthm}

\theoremstyle{plain}
\newtheorem{thm}{Theorem}[section]
\newtheorem*{thm*}{Theorem}
\newtheorem{prop}[thm]{Proposition}
\newtheorem{lem}[thm]{Lemma}
\newtheorem*{prop*}{Proposition}
\newtheorem{cor}[thm]{Corollary}

\newtheorem{porism}[thm]{Porism}

\theoremstyle{definition}
\newtheorem{defn}[thm]{Definition}
\newtheorem{example}[thm]{Example}

\newtheorem{rmk}[thm]{Remark}
\newcommand{\yo}{\text{\usefont{U}{min}{m}{n}\symbol{'210}}}

\DeclareFontFamily{U}{min}{}
\DeclareFontShape{U}{min}{m}{n}{<-> udmj30}{}

\newcommand{\To}{\Rightarrow}

\newcommand{\op}{^\mathrm{op}}
\newcommand{\co}{^\mathrm{co}}

\newcommand{\pow}[1]{\mathcal{P}(#1)}%POWER SET
\renewcommand{\ss}[1]{\{#1\}}%SET CONTAINING

\newcommand{\und}[1]{\langle#1\rangle}

\newcommand{\Mag}{\Psi}
\newcommand{\bp}{\boxplus}
\renewcommand{\phi}{\varphi}

\newcommand{\category}[1]{\mathcal{#1}}

\newcommand{\cC}{\category C}

\newcommand{\K}{\mathbb K}

\newcommand{\PP}{\mathcal{P}}

\newcommand{\sph}{\mathbb{S}}

\newcommand{\Z}{\mathbb Z}
\newcommand{\fun}{\mathbb{F}_1}
\newcommand{\HH}{\hat{H}}

\newcommand{\namedcat}[1]{\mathsf{#1}}
\newcommand{\Ab}{\namedcat{Ab}}

\newcommand{\pFin}{\namedcat{Fin}_\ast}
\newcommand{\pSet}{\namedcat{Set}_\ast}

\newcommand{\CMon}{\namedcat{CMon}}

\newcommand{\Set}{\namedcat{Set}}

\newcommand{\Fin}{\namedcat{Fin}}
\newcommand{\Mod}{\namedcat{Mod}}
\newcommand{\Fun}{\namedcat{Fun}}
\newcommand{\TTop}{\namedcat{Top}}
\newcommand{\sSet}{\namedcat{sSet}}

\newcommand{\psMat}{\namedcat{sMat}_\ast}
\newcommand{\Proj}{\namedcat{Proj}}
\newcommand{\funmod}{\Mod_{\fun}}

\newcommand{\cwumag}{\namedcat{Plas}}
\newcommand{\ccwumag}{\namedcat{Plas}}

\newcommand{\PL}{\cwumag}

\newcommand{\cMsc}{\namedcat{cMsc}}
\newcommand{\cMon}{\namedcat{cMon}}

\title{Projective Geometries as $\fun$-modules}
\author{Jonathan Beardsley and So Nakamura}
\begin{document}
	
	\maketitle
	
	\begin{abstract}
		We describe a fully faithful embedding of the category of axiomatic projective geometries into $\fun$-modules, in the sense of Connes and Consani. In particular, the collineation group of a geometry is isomorphic to the group of automorphisms of its associated $\fun$-module. Along the way, we describe a version of the classifying space construction for a broad generalization of Abelian groups which we call plasmas. Our construction subsumes Segal's existing delooping constructions for Abelian groups, commutative monoids and partial commutative monoids.
	\end{abstract}
	
	\tableofcontents
	
	\subsubsection*{Acknowledgements}
	
	The authors are grateful to Thomas Blom, Sonja Farr, Landon Fox, Philip Hackney, Kiran Luecke, Rajan Mehta, Joe Moeller, Eric Peterson and Manuel Reyes for helpful conversations related to the contents of this paper. The first author was partially supported by a Simons Foundation collaboration grant, Award ID \#853272. An LLM was used to find several errors in the process of revising this paper.
	
	\section{Introduction and Motivation}
	
	In \cite{tits-F1}, Tits suggested that projective geometry over $\mathbb{F}_q$ should become combinatorics, in a sense, as one ``lets $q$ go to $1$.'' The basic dictionary is that $\mathrm{GL}_n(\mathbb{F}_q)$ should become $\Sigma_n$, the symmetric group on $n$ elements, and $\mathbb{F}_q^n$ should become the finite set $\{1,2,3,\ldots,n\}$. A detailed exposition of the motivation behind this intuition can be found in \cite{lorscheid-everyone}. 
	
	If one takes the above analogy literally, then there should be a field with one element $\fun=\{1\}$. Of course, there cannot be such a field. Nonetheless, the question of what $\fun$ ``should be'' has been taken up by many authors. There are so many different approaches to $\fun$ that we could not hope to give an exhaustive bibliography here. For example, in addition to the work of Lorscheid cited above, see \cite{soule-F1varieties, deitmar-schemesoverF1, Durov, toenvaquie-underSpecZ, borger-F1, lorscheid-blueprints_absolute}. Studying $\fun$ has also been motivated by the suggestion, made by Kapranov, Smirnov and others (cf.~\cite{smirnov-hurwitz, manin-zeta-functions}), that a suitable theory of algebraic geometry over $\mathrm{Spec}(\fun)$ (whatever that should be) may lead to a proof of the Riemann Hypothesis. More specifically, because $\Z$ should be an $\fun$-algebra, $\mathrm{Spec}(\Z)$ should be a curve over $\mathrm{Spec}(\fun)$. As a result, Deligne's proof of the Weil Conjectures for function fields might be transferable to this setting.
	
	We do not aim for anything so significant. Instead, the primary goal of this document is to reify Tits' intuition that projective geometry ``lives over $\fun$.'' We accomplish this by showing that the category of projective geometries, realized as abstract systems of axioms in the style of Veblen \cite{veblen1908}, Prenowitz \cite{prenowitz-multigroups} and others, embeds fully faithfully into one of the modern points of view on algebra and geometry over $\fun$. We now briefly describe said point of view.
	
	The work of Connes and Consani identifies ``algebra in characteristic one'' with the study of so-called \textit{Segal $\Gamma$-sets}, i.e.~pointed functors $F\colon \pFin\to\pSet$ from the category of finite pointed sets to the category of all pointed sets \cite{connesconsani-AbsoluteAlgebraSegalGammaRings, GromovNorm}. One of the advantages of this framework is that several other approaches to $\fun$ are subsumed by it, including Durov's generalized rings \cite{Durov} and Deitmar's approach via monoids \cite{deitmar-schemesoverF1}. It also subsumes previous approaches to $\fun$ by Connes and Consani themselves, including via hyperrings \cite{connesconsani-monoidstohyper} and semirings \cite{connesconsani-arithmeticsite}. In this work, we follow Connes and Consani in taking $\Gamma$-sets as our model of $\fun$-modules.
	
	Connes and Consani have specifically considered the relationship between projective geometry and $\fun$, via the theory of hyperrings, in \cite{connesconsanihyperring}, which builds on Prenowitz's approach to projective geometry using multigroups \cite{prenowitz43}. However, this work of Connes and Consani did not connect to Segal's $\Gamma$-sets, and put some restrictions on the allowed geometries (e.g.~they require that each line contain at least four points). Our result holds for a more general class of projective geometries (as in \cite{faure-frolicher-modernPG}) and relates them to the more expansive view on $\fun$-modules afforded by Segal $\Gamma$-sets.
	
	Projective geometry can also be generalized, following Whitney and Nakasawa \cite{whitney-matroids, nishimura-kuroda-nakasawa-matroids} via the theory of \textit{matroids}. Recent work of Reyes and Nakamura \cite{nakamurareyes-mosaics} shows that a nice class of matroids, specifically \textit{simple pointed matroids}, can be faithfully included into the category of so-called \textit{mosaics}. Moreover, this class of matroids contains projective geometries as a full subcategory which is then shown to embed \textit{fully faithfully} into mosaics, via the matroid inclusion.
	
	In this work, we build on the work of \cite{nakamurareyes-mosaics} to show that projective geometries embed fully faithfully into $\Gamma$-sets (Theorem \ref{thm:matroids as funmodules}), which we think of as $\fun$-modules. We do so by first giving a fully faithful embedding of weakly unital commutative hypermagmas, which are a slight generalization of Nakamura and Reyes' mosaics, into $\Gamma$-sets. This embedding recovers Segal's Eilenberg-MacLane embedding of commutative monoids in $\Gamma$-sets and may be of independent interest.
	
	\subsection{Organization of the Paper}
	Section \ref{sec:background} gives background on the main objects of this paper: weakly unital hypermagmas (which we call plasmas) and $\fun$-modules. In Section \ref{subsec:catofplasmas} we introduce the category of \textit{commutative plasmas} which is a generalization of the category of commutative monoids that allows for partially and multiply valued addition. In Section \ref{subsec:plasmas and projs} we recall the work of \cite{nakamurareyes-mosaics}, which embeds projective geometries into commutative plasmas. In Section \ref{subsec:cat of F1 modules}, we recall the category of $\fun$-modules, a.k.a.~Segal $\Gamma$-sets or $\mathfrak{s}$-modules, from the work of Connes and Consani. As motivation for this being the ``correct'' category of $\fun$-modules, we prove that its algebraic $K$-theory is equivalent to the sphere spectrum of stable homotopy theory (Corollary \ref{cor: K of F1}). This result is certainly well known, but we were not able to find (this form of it) it in the literature. 
	
	Section \ref{sec:plasmas and fun modules} obtains the main results of the paper. In Theorem \ref{thm: H is right adjoint to mag} and Corollary \ref{cor:H is fully faithful} we construct an adjunction between $\fun$-modules and commutative plasmas and show that the right adjoint (which we call the plasmic nerve) is a fully faithful inclusion of the latter into the former. This construction is entirely new, and we perform some elementary computations to aid in understanding it. Section \ref{subsec:matroids as fun modules} contains Theorem \ref{thm:matroids as funmodules}, the result which motivated this paper. It uses the work of \cite{nakamurareyes-mosaics} to extend the functor of Section \ref{section:plasmicnerve} to a fully faithful inclusion of projective geometries into $\fun$-modules.
	
	In Section \ref{sec:plasmic nerve in context}, we investigate further properties of the plasmic nerve, including relating it to several similar constructions already existing in the literature. In Section \ref{sec:nerve is corep} we prove Theorem \ref{thm: PM is H} which shows that the plasmic nerve is corepresentable by a plasma valued presheaf on the category of finite pointed sets. This is analogous to the way that the usual nerve of a 1-category is corepresented by a cosimplicial category. This description of the plasmic nerve allows us to give a recognition theorem for its image in $\fun$-modules (Corollary \ref{cor:generalized Segal condition}). This is a generalization of the Segal condition for identifying commutative monoids in the category of $\Gamma$-sets. We make this explicit in Theorem \ref{thm:H hat generalizes H} of Section \ref{sec:nerve and Segal constructions}. In Proposition \ref{prop:nerve give partial monoid nerve} we show that the plasmic nerve recovers Segal's nerve for a commutative partial monoid (and therefore the usual nerve construction for commutative monoids). Finally, in Section \ref{sec:nerves not 2-segal}, we relate our constructions to a similar construction of $\Gamma$-sets from commutative pseudomonoids in spans of sets given in \cite{contreras-mehta-stern}. We show that the two constructions overlap in many places but that our nerve functor is in general more complex, as it allows for non-associative inputs.
	
	We also include Appendix \ref{sec:delooping appendix} as a reminder of how to construct the simplicial set underlying a $\Gamma$-set. This construction is again well known, but we reproduce it here for the sake of completeness. See also \cite[Section 5.2]{contreras-mehta-stern}

	\section{Background}\label{sec:background}
	
	\subsection{The Category of Plasmas}\label{subsec:catofplasmas}
	
	We define a category of rudimentary algebraic objects that we call \textit{plasmas}. The existing technical term for such an object  is \textit{weakly unital commutative hypermagma}. 
	
	\begin{defn}\label{defn:ion}
		For a set $M$, write $\pow M$ for the power set of $M$. A plasma is a triple $(M,\bp,e)$ where $M$ is a set, $\bp$ is a function $M\times M\to \pow M$ and $e$ is an element of $M$, satisfying the conditions that $a\in a\bp e=e\bp a$ and $a\bp b=b\bp a$ for all $a,b\in M$.
		The function $\bp$ is called a \textit{hyperoperation} and $a\bp b$ is called the \textit{hypersum} of the elements $a$ and $b$. We call $e$ the \textit{weak unit} of $M$. When we have need of more than one hyperoperation we will often use $\star$ or $\ast$.
	\end{defn}

	\begin{defn}\label{defn:ionmap}
		A \textit{morphism of plasmas} $(M,\bp,e)\to (M',\star,e')$ is a function $f\colon M\to M'$ satisfying the following two conditions:
		\begin{enumerate}
			\item $f(e)=e'$,
			\item $f(a\bp b)\subseteq f(a)\star f(b)$.
		\end{enumerate}
	\end{defn}
	
	\noindent The following is immediate from the definitions:
	\begin{prop}\label{prop:cat of ions}
		Definitions \ref{defn:ion} and \ref{defn:ionmap} define a category, which we denote $\cwumag$. The category of Abelian groups is a full subcategory of $\cwumag$.
	\end{prop}
	
	Hence we think of plasmas as a generalization of Abelian groups. With slightly different terminology, the degree to which one can do standard algebraic constructions with plasmas, and various of their subcategories, is investigated in \cite{nakamurareyes-mosaics}.

	%\begin{rmk}
	%	They will not appear often in this work, but sometimes we will want to talk about plasmas without the commutativity condition. In that case we will say \textit{noncommutative plasma}.
	%\end{rmk}
	
	\begin{defn} There are certain additional properties we might ask of a plasma $(M,\bp,e)$, which we enumerate below:
		\begin{enumerate} 
			\item We say that a plasma is \textit{associative} if its hyperoperation satisfies the following condition:
			\[
			(a\bp b)\bp c=\bigcup_{x\in a\bp b}x\bp c=\bigcup_{x\in b\bp c}a\bp x=a\bp (b\bp c)
			\]
			\item We call a plasma \textit{strictly unital} if it satisfies $a\bp e=\{a\}$ for all $a\in M$.
			\item A \textit{reversible plasma} is a plasma equipped with a function $(-)^{-1}\colon M\to M$ such that if $a\in b\bp c$ then $b\in a\bp c^{-1}$ and $c\in b^{-1}\bp a$.
			\item A reversible plasma which is strictly unital is called a \textit{commutative mosaic} (following \cite{nakamurareyes-mosaics}).
			\item We say that a plasma is \textit{total} if, for all $a,b\in M$, $a\bp b\neq \varnothing$.
			\item We say that a plasma is \textit{deterministic} if, for all $a,b\in M$, $a\bp b$ is either empty or a singleton. Note that a deterministic plasma is always strictly unital.
			\item A plasma which is total and deterministic is called a \textit{unital commutative magma}. 
			\item A plasma which is total, deterministic, and associative is called a \textit{commutative monoid}.
		\end{enumerate}
	\end{defn}
	
	\begin{rmk}
		By the discussion preceding \cite[Definition 2.3]{nakamurareyes-mosaics}, we know that the ``inverse'' function for mosaics is unique and that a unit preserving morphism of commutative mosaics commutes with inverses. This justifies the first part of the following definition.
	\end{rmk}
	
	\begin{defn}
		We name two full subcategories of $\ccwumag$ that will appear later:
		\begin{enumerate}
			\item $\cMsc\subseteq\ccwumag$ for the full subcategory of commutative mosaics,
			\item $\cMon\subseteq\ccwumag$ for the full subcategory of commutative monoids.
		\end{enumerate}
	\end{defn}
	
	\begin{rmk}
		The category $\cMon$ described above refers to the category of commutative monoid objects in $\Set$ with respect to the Cartesian monoidal structure. We will have occasion to discuss commutative monoid objects in other categories, in which case we will write $\cMon(\cC)$ rather than just $\cMon$.
	\end{rmk}
	
	\subsection{Plasmas, Matroids and Projective Geometries}\label{subsec:plasmas and projs}
	
	In \cite{nakamurareyes-mosaics}, it is shown that projective geometries (in the sense of e.g.~\cite{faure-frolicher-modernPG}) are a full subcategory of commutative mosaics. We recall the basic constructions here and write down the obvious resulting relationship between projective geometries and plasmas.
	
	\begin{defn}
		A \textit{matroid} is a set $M$ with a function $\kappa\colon \mathcal{P}(M)\to \mathcal{P}(M)$, called a closure operator, satisfying the following conditions for all $A,B\subseteq M$ and $x,y\in M$:
		\begin{enumerate}
			\item $A\subseteq \kappa(A)$,
			\item if $A\subseteq B$ then $\kappa(A)\subseteq \kappa(B)$,
			\item $\kappa(\kappa(A))=\kappa(A)$
			\item if $x\in\kappa(A\cup\{y\})$ and $x\notin\kappa(A)$ then $y\in\kappa(A\cup \{x\})$.
		\end{enumerate}
		A morphism of matroids $(M,\kappa)\to (N,\kappa')$ is a function $f\colon M\to N$ such that $f(\kappa(A))\subseteq \kappa'(f(A))$.
		
		We say that a matroid $(M,\kappa)$ is \textit{pointed} if it is equipped with a distinguished element $0\in\kappa(\varnothing)$. We say that a pointed matroid $(M,\kappa,0)$ is \textit{simple} if $\kappa(\varnothing)=\{0\}$ and $\kappa(\{x\})=\{0,x\}$ for all $x\in M$. 
		A morphism of simple pointed matroids is a morphism of matroids that preserves the base point. We denote the category of simple pointed matroids by $\psMat$.
	\end{defn}
	
	\begin{rmk}
		The closure operators of a matroid should be thought of as an operation on a vector space that takes a collection of points in that space to the smallest subspace containing them. In particular, given two points $x,y\in M$, we think of $\kappa(x,y)$ as the span of $x$ and $y$.
	\end{rmk}
	
	\begin{defn}
		A \textit{projective geometry} is a simple pointed matroid $(M,\kappa,0)$ whose closure operator satisfies the following properties:
		\begin{enumerate}
			\item For every $A\subset M$, \[\kappa(A)=\bigcup\{\kappa(B): B\subseteq A\text{~is finite}\}\]
			\item For all $A,B\subseteq M$, \[\kappa(A\cup B)=\bigcup_{x\in \kappa(A)\text{ and }y\in \kappa(B)}\kappa(\{x,y\})\]
		\end{enumerate}
		A morphism of projective geometries is a morphism of pointed matroids.
		We write $\Proj$ for the category of projective geometries (i.e.~the full subcategory of $\psMat$ on the matroids satisfying the above two conditions). 
	\end{defn}
	
	\begin{example}\label{example:trivgeometry}
	The simplest class of projective geometries is that of the \textit{trivial} or \textit{discrete} geometries. Let $\mathcal{P}\und{n}$ denote the power set of $\und{n}$. The trivial geometry on $n$ points is the set $\und{n}$ with closure operator $\kappa\colon \mathcal{P}\und{n}\to\mathcal{P}\und{n}$ given by $\kappa(U)=U\cup\{0\}$. Geometrically, this describes $n$ points such that every pair of points is contained in a line and no line contains more than two points. For later reference, let $\mathbb{G}_n^{triv}$ denote the trivial geometry on $n$ points.
	\end{example}
	
	\begin{rmk}
		There are many different definitions of \textit{projective geometries} in the literature, and the above is non-standard. Without the \textit{pointed} condition, what we call a simple matroid here is called a \textit{combinatorial geometry} in \cite{crapo-rota-combgeom}. It is shown in \cite[Proposition 4.28]{nakamurareyes-mosaics} that the definition of projective geometry given above (which is called a projective pointed matroid there) is equivalent to a more ``standard'' definition of projective geometry from \cite{faure-frolicher-modernPG}.
	\end{rmk}
	
	\begin{thm}[\cite{nakamurareyes-mosaics}]\label{thm:matroids as plasmas}
		There is a faithful functor $\Pi\colon\psMat\to\PL$ which, when restricted to $\Proj$, becomes fully faithful.     
	\end{thm}
	
	\begin{proof}
		To a simple pointed matroid $(M,\kappa, 0)$, one associates a plasma $(M,\bp_\kappa,0)$ in which $0$ is set to be the additive identity and with hyperoperation on $M-\{0\}$ given by the following formula:
		\[
		x\bp_\kappa y=\begin{cases}
			\kappa(x,y)-\{x,y,0\} & x\neq y\\
			\{x,0\} & x=y
		\end{cases}
		\] By \cite[Theorem 4.26]{nakamurareyes-mosaics}, $(M,\bp_\kappa,0)$ is a commutative mosaic and this assignment, which is the identity on morphisms, defines a faithful functor $\psMat\to \cMsc$. By \cite[Theorem 4.30]{nakamurareyes-mosaics}, this functor is fully faithful when restricted to $\Proj\subset \psMat$. Now compose with the fully faithful inclusion $\cMsc\subset \PL$.
	\end{proof}

	\subsection{The Category of $\fun$-modules}\label{subsec:cat of F1 modules}
	We recall the basic definitions of the field with one element, and modules over it, as developed in the work of Connes and Consani \cite{connesconsani-AbsoluteAlgebraSegalGammaRings, GromovNorm, connes-consani-Segal_Gamma_Universal, connes-consani-metaphysics}.

	\begin{defn}
		Let $\pFin$ (resp.~$\pSet$) denote the category whose objects are the sets $\und{n}=\{0,1,2,\ldots,n\}$ with 0 as basepoint, resp.~pointed sets, and pointed functions between them. 
		\begin{enumerate}
			\item We write $\funmod$ for the category $\Fun_\ast(\pFin,\pSet)$ of pointed functors i.e.~functors such that $X(\und{0})$ is a singleton set, and call this the \textit{category of $\fun$-modules}. 
			\item Given an $\fun$-module $X$, we write $X_n$ for its value at $\und{n}$. For a morphism $\phi\colon \und{n}\to\und{m}$ of $\pFin$ we write $\phi^X$ for $X(\phi)$.
			\item We write $\fun$ for the inclusion functor $\pFin\hookrightarrow\pSet$.
		\end{enumerate}
	\end{defn}
	
	\begin{rmk}
		In work of Connes and Consani, the functor $\fun$ is often written as $\mathfrak{s}$ or $\mathbb{S}$ (despite playing the role of a ``field with one element'' in their theory). We avoid this notation here because $\fun$ is very much \textit{not} the homotopy theorist's sphere spectrum. Indeed, it is not a spectrum at all, as it does not satisfy the Segal condition.
		
		It is true that some homotopy theorists (e.g.~\cite{localKtheory}) use $\mathbb{S}$ to denote what we call $\fun$ here. However, in that setting they are relying on the fact that the \textit{fibrant replacement} of what we call $\fun$, with respect to a certain Quillen model structure on $\Fun_\ast(\pFin,\TTop_\ast)$, is a model for the classical sphere spectrum. This is one way of saying that the derived group completion of $\fun$, i.e.~$K(\fun)$, is equivalent to the sphere spectrum. Another way to say this is Corollary \ref{cor: K of F1} below.
	\end{rmk}
	
	\begin{rmk}
		Note that $\fun$, as defined above, is the functor corepresented by $\und{1}=\{0,1\}$. We will consider functors corepresented by $\und{n}$ for general $n$ in Section \ref{sec:nerve is corep}.
	\end{rmk}
	
	\begin{rmk}
		The category $\pFin$ is often denoted $\Gamma\op$, because it is equivalent to the opposite of the category $\Gamma$ introduced in Segal's seminal work \cite{segal}. We use $\pFin$ here however because the notation is more descriptive and we can avoid having to write $\op$ quite as often.
	\end{rmk}
	
	\begin{defn}\label{defn:maps of gamma2}
		We will repeatedly need to refer to certain morphisms of $\pFin$, so we give them dedicated names here. We write:
		\begin{enumerate}
			\item $\rho_i$ for any function $\und{n}\to\und{1}$ which takes $i$ to $1$ and every other element to $0$.
			\item $\alpha$ for the function $\und{2}\to\und{1}$ which takes both $1$ and $2$ to $1$.
			\item $i_1$ and $i_2$ for the functions $\und{1}\to\und{2}$ which take $1$ to $1$ and $2$ respectively.
			% \item $\epsilon$ for the map $\und{1}\to\und{2}$ which takes $1$ to $0$.
			% \item $\eta$ for the map $\und{2}\to\und{1}$ which takes both $1$ and $2$ to $0$.
			\item $\tau$ for the non-trivial permutation $\und{2}\to\und{2}$.
			\item $e$ for the unique map $\und{0}\to\und{1}$. %(note that this does not have a representation in the notation of \cref{defn:notation for gamma maps}). 
			\item $\zeta$ for the unique map $\und{1}\to\und{0}$.
			
		\end{enumerate}
		If $X\colon \pFin\to\pSet$ is an $\fun$-module then we will write $\sigma^X_2\colon X_2\to X_1\times X_1$ for the map induced by $\rho_1$ and $\rho_2$. If $X$ is an $\fun$-module, we will write $e^X$ for both the map $\{\ast\}=X_0\to X_1$ and the element in its image. In general, if $X$ is an $\fun$-module and $\phi$ is a morphism in $\pFin$ we will write $\phi^X$ rather than $X(\phi)$.
	\end{defn}
	
	The category of $\fun$-modules has a symmetric monoidal structure called \textit{Day convolution} with respect to the smash product monoidal structure on both $\pFin$ and $\pSet$. The following is a standard result. See for instance Section 1.2.5 of \cite{localKtheory}. It justifies the terminology ``$\fun$-modules.''
	
	\begin{thm}
		The category $\funmod$ admits a symmetric monoidal structure, the \textit{Day convolution}, with respect to which $\fun$ is the monoidal unit. Algebras for this structure are precisely functors which are lax monoidal with respect to the smash product monoidal structures on both $\pFin$ and $\pSet$.
	\end{thm}

	It is a typical desideratum of any model of $\fun$ that its algebraic $K$-theory be equivalent to the sphere spectrum $\mathbb{S}$.  One answer to this question has been described already, in which a kind of group completion (fibrant replacement in the stable model structure) produces $\mathbb{S}$ from $\fun$. Alternatively, one might take the point of view that the algebraic $K$-theory of a ring $R$ should be the homotopy groups of the space $\operatorname{BGL}_\infty(R)^+$ as in \cite{quillen-plusconst}. The general linear groups of $\fun$ are easy to compute, and we do so below.
	\begin{prop}\label{prop:GLnF1}
		Given an $\fun$-module $X\colon \pFin\to\pSet$, write $\operatorname{GL}_n(X)$ for the group of natural automorphisms $\bigoplus_n X\To \bigoplus_n X$ in $\Mod_{\fun}$ (where the coproduct is in the functor category, hence taken pointwise). Then $\operatorname{GL}_n(\fun)$ is isomorphic to $\Sigma_n$, the symmetric group on $n$ letters.
	\end{prop}
	
	\begin{proof}
		A natural isomorphism $\eta\colon\coprod_n\fun\To\coprod_n\fun$ corresponds to a system of pointed functions $\eta_k\colon \bigvee_n\und{k}\to\bigvee_n\und{k}$ for all $k$ with the following property: for any morphism $\phi\colon\und{k}\to \und{j}$ in $\pFin$, the following diagram commutes:
		% https://q.uiver.app/#q=WzAsNCxbMCwwLCJcXGJpZ3ZlZV9uXFx1bmR7a30iXSxbMSwwLCJcXGJpZ3ZlZV9uXFx1bmR7a30iXSxbMCwxLCJcXGJpZ3ZlZV9uXFx1bmR7an0iXSxbMSwxLCJcXGJpZ3ZlZV9uXFx1bmR7an0iXSxbMCwyLCJcXHZlZV9uXFxwaGkiLDJdLFsxLDMsIlxcdmVlX25cXHBoaSJdLFswLDEsIlxcZXRhX2siXSxbMiwzLCJcXGV0YV9qIiwyXV0=
		\[\begin{tikzcd}
			{\bigvee_n\und{k}} & {\bigvee_n\und{k}} \\
			{\bigvee_n\und{j}} & {\bigvee_n\und{j}}
			\arrow["{\vee_n\phi}"', from=1-1, to=2-1]
			\arrow["{\vee_n\phi}", from=1-2, to=2-2]
			\arrow["{\eta_k}", from=1-1, to=1-2]
			\arrow["{\eta_j}"', from=2-1, to=2-2]
		\end{tikzcd}\]
		Let $\sigma\in \Sigma_n$ and write $(i,q)$ for the element of $\bigvee_n\und{k}$ corresponding to $i\in\und{k}$ in the $q^{th}$ summand. Define $\sigma_k\colon \bigvee_n\und{k}\to \bigvee_n\und{k}$ by $\sigma(i,q)=(i,\sigma(q))$. In other words, $\sigma$ permutes the summands. This is clearly natural and defines a group homomorphism $\Sigma_n\to \operatorname{GL}_n(\fun)$.
		
		Now consider the projection $\rho_i\colon \und{k}\to\und{1}$ which takes $i\in\und{k}$ to $1$ and takes every other element to $0$. If $\eta$ is to be a natural isomorphism then the following diagram must commute:
		% https://q.uiver.app/#q=WzAsNCxbMCwwLCJcXGJpZ3ZlZV9uXFx1bmR7a30iXSxbMSwwLCJcXGJpZ3ZlZV9uXFx1bmR7a30iXSxbMCwxLCJcXGJpZ3ZlZV9uXFx1bmR7MX0iXSxbMSwxLCJcXGJpZ3ZlZV9uXFx1bmR7MX0iXSxbMCwyLCJcXHZlZV9uXFxyaG9faSIsMl0sWzEsMywiXFx2ZWVfblxccmhvX2kiXSxbMCwxLCJcXGV0YV9rIl0sWzIsMywiXFxldGFfaiIsMl1d
		\[\begin{tikzcd}
			{\bigvee_n\und{k}} & {\bigvee_n\und{k}} \\
			{\bigvee_n\und{1}} & {\bigvee_n\und{1}}
			\arrow["{\vee_n\rho_i}"', from=1-1, to=2-1]
			\arrow["{\vee_n\rho_i}", from=1-2, to=2-2]
			\arrow["{\eta_k}", from=1-1, to=1-2]
			\arrow["{\eta_1}"', from=2-1, to=2-2]
		\end{tikzcd}\]
		Consider $(i,q)\in\bigvee_n\und{k}$ and suppose $\eta_k(i,q)=(j,p)$ for some $j\neq i$. Then $\vee_n\rho_i\circ\eta_k(i,q)=0$. But $\vee_n\rho_i(i,q)=(1,q)\neq 0$ and $\eta_1$ is a pointed isomorphism, so $\eta_1(1,q)\neq 0$. Therefore for each $i\in \und{k}$ it must be the case that $\eta_k(i,q)=(i,p)$ for some $p$. 
		
		It remains to show that $\eta_k(i,q)=(i,p)$ then $\eta_k(j,p)=(i,p)$ for every $j$.
		This follows from a similar argument after replacing $\rho_i$ with a permutation $\tau\in\Sigma_k$ thought of as a morphism $\und{k}\to\und{k}$.
	\end{proof}
	
	Now we may apply the Barratt-Priddy-Quillen Theorem (\cite[Section 4]{barratt-priddy-monoids}, \cite[Proposition 3.5]{segal}) to obtain:
	
	\begin{cor}\label{cor: K of F1}
		There is an equivalence of spectra $K(\fun)\simeq\sph$, where $\sph$ denotes the sphere spectrum.
	\end{cor}
	
	\section{Plasmas and $\fun$-modules}\label{sec:plasmas and fun modules}
	We now show that plasmas embed fully faithfully in $\fun$-modules via a functor we call the plasmic nerve. It follows from this that projective geometries embed fully faithfully in $\fun$-modules (and pointed simple matroids as well, but only faithfully). This all occurs in Section \ref{section:plasmicnerve}. However we investigate our nerve functor further in the later sections. We believe that, as a direct generalization of both Segal's classical fully faithful embedding $\cMon\hookrightarrow \funmod$ \cite{segal} and Segal's nerve of a partial monoid \cite{segal-config_spaces}, it is of independent interest.

	\subsection{The Nerve of a Plasma}\label{section:plasmicnerve}
	
	The nerve functor for plasmas will be right adjoint to a kind of ``truncation'' functor which we now describe.
	
	\begin{defn}\label{defn:truncated plasma structure}
		Let $X\colon \pFin\to \pSet$ be an $\fun$-module. We write $(X_1,\bp_X,e^X)$ for the set $X_1=X(\und{1})$ equipped with the following data:
		\begin{enumerate}
			\item The element $e^X$ is the image in $X_1$ of the morphism $\{\ast\}=X_0\to X_1$.
			\item The function $\bp_X\colon X_1\times X_1\to \pow{X_1}$ is defined by $(x,y)\mapsto \alpha^X(\sigma_2^X)^{-1}(x,y)$.
		\end{enumerate}
	\end{defn}
	
	\begin{prop}
		If $X\colon \pFin\to \pSet$ is an $\fun$-module then $(X_1,\bp_X,e^X)$ is a commutative plasma. 
	\end{prop}
	\begin{proof}
		We have defined a function $X_1\times X_1\to \pow{X_1}$, so it remains to show that it is commutative and that $e^X$ is a weak unit for it. Note that in $\pFin$ we have the equalities $\alpha\circ\tau=\alpha$, $\rho_1\circ\tau=\rho_2$, and $\rho_2\circ\tau=\rho_1$. Let $z\in x\bp_Xy$, i.e.~$z=\alpha^X(z')$ for some $z'$ with $\rho_1^X(z')=x$ and $\rho_2^X(z')=y$. Then $z=\alpha^X(z')=\alpha^X\circ\tau^X(z')$, where $\rho_1^X(\tau^X(z'))=y$ and $\rho_2^X(\tau^X(z'))=x$. It follows that $z\in y\bp_X x$. The reverse argument is identical, so $x\bp_X y=y\bp_X x$.
		
		Now it is necessary to show that for all $x\in X_1$, we have $x\in x\bp_X e^X$. Note that there is an equality $\rho_2\circ i_1=e\circ\zeta$. It follows that for every $x\in X_1$ we have $\rho_2^X\circ i_1^X(x)=e^X$. Similarly, $\rho_1^X\circ i_1^X(x)=id_{X_1}(x)=x$. Hence $i_1^X(x)\in(\sigma_2^X)^{-1}(x,e^X)$ and the result follows from noticing that $\alpha\circ i_1=id_{\und{1}}$.
	\end{proof}
	
	\begin{prop}
		If $f\colon X\to Y$ is a morphism of $\fun$-modules then $f_1\colon X_1\to Y_1$ is a morphism of commutative plasmas with respect to the plasma structure of Definition \ref{defn:truncated plasma structure}.
	\end{prop}

	\begin{proof}
		The result follows from studying the following diagram: 
		
		% https://q.uiver.app/#q=WzAsNyxbMCwzLCJYXzEiXSxbMiwzLCJZXzEiXSxbMiwxLCJZXzIiXSxbMCwxLCJYXzIiXSxbMCwwLCJYXzFcXHRpbWVzIFhfMSJdLFsyLDAsIllfMVxcdGltZXMgWV8xIl0sWzEsNCwiXFx7XFxhc3RcXH0iXSxbMCwxLCJmXzEiXSxbMywwLCJcXHJob18xXlgiLDIseyJvZmZzZXQiOjEsImN1cnZlIjoxfV0sWzMsMCwiXFxyaG9fMl5YIiwwLHsib2Zmc2V0IjotMSwiY3VydmUiOi0xfV0sWzMsMCwiXFxtdV5YIiwxXSxbMiwxLCJcXG11XlkiLDFdLFsyLDEsIlxccmhvXzFeWSIsMix7Im9mZnNldCI6MSwiY3VydmUiOjF9XSxbMiwxLCJcXHJob18yXlkiLDAseyJvZmZzZXQiOi0xLCJjdXJ2ZSI6LTF9XSxbMywyLCJmXzIiXSxbMyw0LCJcXHNpZ21hX1giXSxbMiw1LCJcXHNpZ21hX1kiLDJdLFs0LDUsImZfMVxcdGltZXMgZl8xIl0sWzYsMCwiZV9YIl0sWzYsMSwiZV9ZIiwyXV0=
		\[\begin{tikzcd}
			{X_1\times X_1} && {Y_1\times Y_1} \\
			{X_2} && {Y_2} \\
			\\
			{X_1} && {Y_1} \\
			& {\{\ast\}}
			\arrow["{f_1}", from=4-1, to=4-3]
			\arrow["{\rho_1^X}"', shift right=1, curve={height=6pt}, from=2-1, to=4-1]
			\arrow["{\rho_2^X}", shift left=1, curve={height=-6pt}, from=2-1, to=4-1]
			\arrow["{\alpha^X}"{description}, from=2-1, to=4-1]
			\arrow["{\alpha^Y}"{description}, from=2-3, to=4-3]
			\arrow["{\rho_1^Y}"', shift right=1, curve={height=6pt}, from=2-3, to=4-3]
			\arrow["{\rho_2^Y}", shift left=1, curve={height=-6pt}, from=2-3, to=4-3]
			\arrow["{f_2}", from=2-1, to=2-3]
			\arrow["{\sigma_2^X}", from=2-1, to=1-1]
			\arrow["{\sigma_2^Y}"', from=2-3, to=1-3]
			\arrow["{f_1\times f_1}", from=1-1, to=1-3]
			\arrow["{e^X}", from=5-2, to=4-1]
			\arrow["{e^Y}"', from=5-2, to=4-3]
		\end{tikzcd}\]
		It is immediate that $f_1(e^X)=e^Y$ so that $f_1$ preserves units. We now check that $f_1(x\bp_X y)\subseteq f_1(x)\bp_Y f_1(y)$. Let $(x,y)\in X_1\times X_1$. It is an easy diagram chase to check that $f_2(\sigma_2^X)^{-1}(x,y)\subseteq (\sigma_2^Y)^{-1}(f_1(x),f_1(y))$. We want to show that $f_1\alpha^X(\sigma_2^X)^{-1}(x,y)\subseteq \alpha^Y(\sigma_2^Y)^{-1}(f_1(x),f_1(y))$. By commutativity, the left hand set is equal to $\alpha^Yf_2(\sigma_2^X)^{-1}(x,y)$ which by the preceding argument is a subset of $\alpha^Y(\sigma_2^Y)^{-1}(f_1(x),f_1(y))$.
	\end{proof}
	
	The above two propositions assemble to justify the following definition.
	
	\begin{defn}
		Let $\Mag\colon \funmod\to \ccwumag$ be the functor which takes an $\fun$-module $X$ to $(X_1,\bp_X,e^X)$ and a morphism $f\colon X\to Y$ of $\fun$-modules to $f_1\colon X_1\to Y_1$.
	\end{defn}
	
	\begin{example}\label{example:mag of F1}
		The underlying plasma of $\fun$, i.e.~$\Mag$ applied to the inclusion $\pFin\hookrightarrow \Set_\ast$, is the set $\{0,1\}$ with hyperoperation $0\bp 0=0$, $1\bp 0=0\bp 1=1$ and $1\bp 1=\varnothing$. One can check that $\Mag\fun$ is the free plasma containing a non-zero element.
	\end{example}
	
	\begin{example}
		Let $HA\colon \pFin\to\pSet$ be the $\fun$-module associated to a commutative monoid $A$ by \cite[Section 1]{segal} (or equivalently \cite[Section 2.2]{connesconsani-AbsoluteAlgebraSegalGammaRings}). Then $\Mag HA$ is simply $A$ itself. See Section \ref{sec:nerve and Segal constructions} and Definition \ref{defn:SegalHfunctor} for more on this.
	\end{example}
	
	We now describe the functor that will be right adjoint to $\Mag$.
	
	\begin{defn}\label{defn: H on objects}
		Write $\pow n$ for the power set of $[n]=\und{n}-\{0\}$. Let $(M,\boxplus, 0)$ be a commutative plasma and let $n\geq 0$. Define a pointed set by
		\[
		\HH M_n=\left\{(x_S)_{S\subseteq[n]}\in M^{\pow n}:x_{\varnothing}=0,~x_{S\cup T}\in x_S\boxplus x_T~\text{whenever}~S\cap T=\varnothing\right\}
		\] Sometimes we will write an element $(x_S)_{S\subseteq[n]}$ as $(x_S)_{\pow n}$. We will frequently want to explicitly write elements of $\HH M_n$ for low $n$. When doing so, we will not write the leading zero, as it does not add any real information to the set.
	\end{defn}

	\begin{defn}\label{defn: H on morphisms}
		Let $\phi\colon \und{n}\to\und{m}$ be a morphism in $\pFin$. Then we define the function of pointed sets $\HH M(\phi)\colon \HH M_n\to M^{\pow m}$ as follows. Let $(x_S)_{\pow n}\in \HH M_n\subseteq M^{\pow n}$ and let $T\subseteq [m]$. Then 
		\[
		\HH M(\phi)(x_S)_{\pow n}=\left(x_{\phi^{-1}(T)}\right)_{\pow m}\in M^{\pow m}
		\]
	\end{defn}
	
	\begin{lem}
		For $\phi\colon \und{n}\to\und{m}$ in $\pFin$, the function $\HH M(\phi)$ factors through $\HH M_m$ and preserves the basepoint. 
	\end{lem}
	
	\begin{proof} It is immediate that $\HH M(\phi)$ preserves the basepoint which is 0 in every coordinate. Now Suppose that $T\subseteq [m]$, $T=T_1\cup T_2$ and $T_1\cap T_2=\varnothing$. We must show that $x_{\phi^{-1}(T)}\in x_{\phi^{-1}(T_1)}\boxplus x_{\phi^{-1}(T_2)}$. But of course $\phi^{-1}(T)=\phi^{-1}(T_1\cup T_2)=\phi^{-1}(T_1)\cup \phi^{-1}(T_2)$. Moreover, since $T_1\cap T_2=\varnothing$, we have that $\phi^{-1}(T_1)\cap \phi^{-1}(T_2)=\varnothing$. Now because $(x_S)_{\pow n}\in \HH M_n$, it must be true that $x_{\phi^{-1}(T)}\in x_{\phi^{-1}(T_1)}\boxplus x_{\phi^{-1}(T_2)}$.
	\end{proof}
	
	\begin{cor}
		Let $(M,\bp,0)$ be a plasma. Then Definitions \ref{defn: H on objects} and \ref{defn: H on morphisms} define a functor $\HH M\colon \pFin\to \pSet$.
	\end{cor}
	
	\begin{defn}
		Let $f\colon M\to M'$ be a morphism in $\ccwumag$. Then define maps $\HH f_n\colon \HH M_n\to \HH M'_n$ as the postcomposition map $(x_S)_{\pow n}\mapsto (f(x_S))_{\pow n}$.
	\end{defn}
	
	From the above definitions it is straightforward to conclude the following:
	
	\begin{prop}
		The above definitions assemble into a functor $\HH\colon \PL\to \funmod$.
	\end{prop}
	
	Now we show that our functor $\hat{H}$ is the correct one.
	\begin{thm}\label{thm: H is right adjoint to mag}
		The functor $\hat{H}$ is right adjoint to $\Mag$. 
	\end{thm}
	\begin{proof}
		Let $X$ be an $\fun$-module and $(M,\star,0)$ a plasma. We define a map
		\[\eta\colon \ccwumag(\Mag X, M)\rightarrow \funmod(X, \hat{H}M)\]
		by sending a morphism $f\colon\Mag X=X_1\rightarrow M$ of $\ccwumag$ to a natural transformation $\eta f\colon X_n\rightarrow\HH M_n$ defined by the formula  \[x\mapsto (f\circ \rho^X_S(x))_{S\subset [n]}\] in which $\rho_S:\langle n\rangle\rightarrow\langle 1\rangle$ is the unique map such that $\rho^{-1}_S(1)=S$. 
		
		We begin by showing that this function is well defined. Note that because $f$ is a morphism of plasmas, we have \[f(x\bp_X y)=f(\alpha^X(\sigma_2^X)^{-1}(x,y))\subseteq f(x)\star f(y)\] for all $x,y\in X_1$. We must show that if $S,T\subseteq[n]$ and $S\cap T=\varnothing$ then $f\circ\rho_{S\cup T}^X(x)\in (f\circ\rho_S^X(x))\star (f\circ\rho_T^X(x))$. Write $\rho_{ST}\colon \und{n}\to\und{2}$ for the function with $\rho_{ST}^{-1}(1)=S$ and $\rho_{ST}^{-1}(2)=T$. It follows that $\rho_S=\rho_1\circ\rho_{ST}$, $\rho_T=\rho_2\circ\rho_{ST}$ and $\rho_{S\cup T}=\alpha\circ\rho_{ST}$. Therefore we have obtained
		\[\rho^X_{S\cup T}(x)\in \rho^X_S(x)\bp_X \rho^X_T(x)\]
		for all $x\in X_n$. Because $f$ is a morphism of plasmas, we therefore have \[f\circ\rho_{S\cup T}^X(x)\in f\left(\rho_S^X(x)\bp_X\rho_T^X(x)\right)\subseteq \left(f\circ\rho_S^X(x)\right)\star \left(f\circ\rho_T^X(x)\right)\] as desired.
		
		The naturality of maps follows from the fact that $\rho_{f^{-1}(S)}=\rho_S\circ f$ for all subset $S\subset [n']$ and a morphism $f:\langle n\rangle\rightarrow\langle n'\rangle$ of pointed sets. Therefore the assignments indeed define a natural transformation from $X$ to $\hat{H}M$.
		
		We show that $\eta$ is the inverse of 
		\[\nu\colon \funmod(X, \HH M)\rightarrow \ccwumag(\Psi X, M)\]
		given by sending a natural transformation $g=\{g_n\}$ to $\Mag g=g_1$. By construction, we have $\nu\circ\eta=id$. Therefore it suffices to show that $\nu$ is injective.
		
		Let $g=\{g_n\}, g'=\{g'_n\}:X\rightarrow \hat{H}M$ be natural transformations such that $\nu(g)=\nu(g')$, i.e., $g_1=g'_1$. We show that $g_n(x)=(g_1\circ\rho^X_S(x))_{S\subset [n]}\in\HH M_n$ for all $x\in X_n$. Indeed,  we have that
		$$\hat{H}M\rho_T((m_S)_S)_{\{1\}}=m_{\rho^{-1}_T(1)}=m_T$$
		for all $(m_S)_S\in\hat{H}M\langle n\rangle$. The naturality $g_1\circ \rho^X_T=\hat{H}M\rho_T\circ g_n$ implies that $g_n(x)_S=g_1(\rho^X_S(x))$ for all $S\subset [n]$. By the same argument we obtain $g'_n(x)=(g'_1\circ\rho^X_S(x))_{S\subset [n]}\in\hat{H}M\langle n\rangle$ for all $x\in X_n$. Since $g_1=g'_1$, we have $g=g'$ and the proof is complete.
	\end{proof}

	This next statement follows immediately from the proof of Theorem \ref{thm: H is right adjoint to mag} along with standard facts about adjunctions.
	
	\begin{porism}
		For any commutative plasma $M$, the counit of the $\Mag \dashv \HH$ adjunction, $\Mag\HH M\to M$, is the identity. For any $\fun$-module $X$, the unit $X\to \HH\Mag X$ is given in degree $n$ by the function \[X_n\ni z\mapsto (\rho_S^X(z))_{S\subseteq[n]}\in (\HH\Mag X)_n\]
	\end{porism}
	
	% \begin{lem}
		%     For any $\fun$-module $X$ and any $n$, the unit map $X_n\to (\HH\Mag X)_n$ is surjective.
		% \end{lem}
	
	% \begin{proof}
		%     An element of $(\HH\Mag X)_n$ is a function $f\colon \pow n\to X_1$ with the property that for every $S,U\in\pow n$, whenever $S\cap U=\varnothing$,  we have 
		% \end{proof}
	
	\begin{cor}\label{cor:H is fully faithful}
		The right adjoint $\HH$ is fully faithful.
	\end{cor}
	
	% It will also be useful to record the following:
	% \begin{lem}
		%     Let $X$ be a $\fun$-module. The morphism $X_n\to (\HH\Mag X)_n$ given by the unit of the $\Mag\dashv \HH$ adjunction evaluated at $n$, is a surjection. 
		% \end{lem}
	
	\begin{example}
		To assist with understanding, we write down the formulas that determine the value of $\HH$ at $\und{3}$ and $\und{4}$ explicitly. Note that \[\pow 3=\{\varnothing, \ss{1},\ss{2},\ss{3},\ss{1,2},\ss{1,3},\ss{2,3},\ss{1,2,3}\}.\] It follows that elements of $\HH M_3$ are 8-tuples $(a_0,a_1,a_2,a_3,a_{12},a_{13},a_{23},a_{123})$ satisfying the following relations:
		\begin{multicols}{2}
			\begin{enumerate}
				\item $a_0=0$.
				\item $a_{12}\in a_1\boxplus a_2$.
				\item $a_{13}\in a_1\boxplus a_3$.
				\item $a_{23}\in a_2\boxplus a_3$. 
				\item $a_{123}\in a_1\boxplus a_{23}$.
				\item $a_{123}\in a_2\boxplus a_{13}$.
				\item $a_{123}\in a_3\boxplus a_{12}$. 
			\end{enumerate}
		\end{multicols}
		Note that there are a swathe of ``union'' relations involving $a_0=a_\varnothing=0$ that are automatically satisfied by $M$ being weakly unital. For instance, we know that $\{1\}\cup\varnothing=\{1\}$ and $\{1\}\cap \varnothing=\varnothing$, but it is also true that $a_1\in a_1\boxplus a_0$ by weak unitality. 
		
		For $\HH M_4$ we'll have 16-tuples of the form\[
		(a_0,a_1,a_2,a_3,a_4,a_{12},a_{13},a_{14},a_{23},a_{24},a_{34},a_{123},a_{124},a_{134},a_{234},a_{1234})
		\]
		satisfying the following conditions:
		\begin{enumerate}
			\item $a_0=0$.
			\item $a_{12}\in a_1\boxplus a_2$.
			\item $a_{13}\in a_1\boxplus a_3$.
			\item $a_{14}\in a_1\boxplus a_4$.
			\item $a_{23}\in a_2\boxplus a_3$.
			\item $a_{24}\in a_2\boxplus a_4$.
			\item $a_{34}\in a_3\boxplus a_4$.
			\item $a_{123}\in (a_1\boxplus a_{23})\cap  (a_2\boxplus a_{13})\cap (a_3\boxplus a_{12})$. 
			\item $a_{124}\in (a_1\boxplus a_{24})\cap  (a_2\boxplus a_{14})\cap (a_4\boxplus a_{12})$.
			\item $a_{134}\in (a_1\boxplus a_{34})\cap  (a_3\boxplus a_{14})\cap (a_4\boxplus a_{13})$.
			\item $a_{234}\in (a_2\boxplus a_{34})\cap  (a_3\boxplus a_{24})\cap (a_4\boxplus a_{23})$.
			\item $a_{1234}\in (a_1\boxplus a_{234})\cap (a_2\boxplus a_{134})\cap (a_3\boxplus a_{124})\cap (a_4\boxplus a_{123})\cap (a_{12}\boxplus a_{34})\cap (a_{13}\boxplus a_{24})\cap (a_{14}\boxplus a_{23})$.
		\end{enumerate}
		
	\end{example}
	
	The following is easy to see, but helps illustrate the nature of the functor $\HH$:
	
	\begin{prop}\label{prop: H at 012}
		If $(M,\bp,0)$ is a commutative plasma then there are isomorphisms:
		\begin{align*}
			\HH M_0&\cong \{0\}\\
			\HH M_1&\cong M\\
			\HH M_2&\cong\{(a,b,c)\in M^3:c\in a\bp b\}
		\end{align*}
		Moreover, for any $(a,b,c)\in\HH M_2$, we have $
		\rho_1^{\HH M}(a,b,c)=a$, 
		$\rho_2^{\HH M}(a,b,c)=b$, and 
		$\alpha^{\HH M}(a,b,c)=c$.
	\end{prop}

	Proposition \ref{prop: H at 012} indicates that $\HH M_2$ encodes all possible hyperadditions that can be performed in $M$ (which corresponds to measuring how surjective $\sigma_2^{\HH M}$ is) but ``remembers'' every possible output of each of those additions (correponding to the degree to which $\sigma_2^{\HH M}$ injects into its image). Given a pair $(a,b)\in M^2$ and some $c\in a\bp b$, the map $\alpha^{\HH M}$ ``performs'' the addition.
	
	In higher degrees, $\HH M$ also records all possible hypersums and all possible outputs which are in the intersection of all possible associations of the inputs. However, to be compatible with \textit{all} the morphisms in $\pFin$, $\HH M$ has to record all possible ``paths'' from an $n$-tuple $(x_1,\ldots,x_n)$ to any of its outputs. When $M$ is deterministic, e.g.~$M$ is a partial monoid, this data is all unique and so, up to isomorphism, disappears. However, when $M$ has a multi-valued hyperoperation, it's possible for there to be many  inequivalent ways to take a sum $x_1\bp\cdots\bp x_n$ and produce the same output. In particular, although every output must be in the intersection of every parenthesization of $x_1\bp\cdots\bp x_n$, different parenthesizations may result in different ``paths'' to that output.  We will see one consequence of the fact that this data is recorded in Section \ref{sec:nerves not 2-segal}.

	\begin{example}
		Consider the plasma $\Mag{\fun}$ of Example \ref{example:mag of F1} whose elements are $\{0,1\}$ and whose hyperaddition is given by $0\boxplus 0=\{0\}$, $1\boxplus 0=0\boxplus 1=1$ and $1\boxplus 1=\varnothing$. We list the first few values of $\HH \Mag\fun$:
		\begin{enumerate}
			\item $\HH M_0=\{0\}$.
			\item $\HH M_1=\{0,1\}$.
			\item $\HH M_2=\{(0,0,0), (0,1,1),(1,0,1)\}$
			\item 
			$
			\HH M_3=
			\{(0,0,0,0,0,0,0), (0,0,1,0,1,1,1),
			(0,1,0,1,0,1,1),\\ (1,0,0,1,1,0,1)\}$
		\end{enumerate}
		In general, $(\HH\Mag\fun)_n$ is naturally isomorphic to $\und{n}$ for all $n$.
	\end{example}
	
	\begin{prop}
		There is a natural isomorphism $\HH\Mag\fun\cong\fun$.
	\end{prop}
	
	\begin{proof}
		Note that because $\Mag\fun$ is deterministic, an element of $(\HH\Mag\fun)_n$ is the same data as an $n$-tuple in $\{0,1\}^n$ in which $1$ appears at most once (otherwise the addition is not defined). When 1 does not appear at all, we have the basepoint, and otherwise we have the elements corresponding to $[n]\subseteq\und{n}$. It is a routine check to show naturality.
	\end{proof}
	
	\subsection{Matroids and Projective Geometries as $\fun$-modules}\label{subsec:matroids as fun modules}
	
	It now follows directly that we have a faithful functor from simple pointed matroids to $\funmod$ and that, through this, projective geometries embed fully faithfully into $\fun$-modules.
	
	\begin{thm}\label{thm:matroids as funmodules}
		There is a faithful functor $\HH\Pi\colon \psMat\to \funmod$ which is fully faithful on the full subcategory $\Proj$ of projective geometries.
	\end{thm}
	
	\begin{proof}
		Because faithful functors, and fully faithful functors, are closed under composition, this follows immediately from Theorem \ref{thm:matroids as plasmas} and Corollary \ref{cor:H is fully faithful}.
	\end{proof}
	
	\begin{example}
	Note that the plasma $\Psi\fun$, whose elements are $\{0,1\}$ equipped with $1\boxplus 1=\varnothing$, cannot be $\Pi G$ for any projective geometry $G$. Any plasma obtained from a projective geometry must satisfy $x\boxplus x=\{x,0\}$. This implies that $\fun$ itself cannot be the $\fun$-module associated to any projective geometry.
	\end{example}

	\begin{example}
	The plasmas and $\fun$-modules associated to the trivial geometries $\mathbb{G}_n^{triv}$ of Example \ref{example:trivgeometry} have been computed in \cite{beardsley-dynkin}. For $n=1$ we have $\Pi\mathbb{G}_1^{triv}=\mathbb{K}$, i.e.~the plasma with elements $\{0,1\}$ and hyperoperation $1\boxplus 1=\{0,1\}$. In general, $\Pi\mathbb{G}_n^{triv}$ has elements $\{0,1,\ldots,n\}$, with $0$ as strict identity element. For $a$ and $b$ non-zero the commutative hyperoperation is given by
\[ a\boxplus b =\begin{cases} \{0,a\} & a=b\\ \varnothing & a\neq b\end{cases}\]
 We denote this plasma by $\vee_n\mathbb{K}$ though it is not the categorical coproduct in $\cwumag$. It is straightforward to check that $\hat{H}(\vee_n\mathbb{K})\cong \vee\hat{H}(\mathbb{K})$. As such, it suffices to compute $\hat{H}\mathbb{K}$. It follows from \cite[Theorem 4.7]{beardsley-dynkin} that $\hat{H}\mathbb{K}_n$ is the set of so-called \textit{Dynkin systems} on $\langle n\rangle$ or, equivalently, the set of \textit{$\lambda$-algebra structures} (in the sense of measure theory) on $\langle n\rangle$. These structures are investigated at length in \cite{derr-williamson-dynkin}.
	\end{example}

	\section{The Plasmic Nerve in Context}\label{sec:plasmic nerve in context}
	
	We now show that the nerve functor $\HH$ directly generalizes the classical construction of a $\Gamma$-set from a commutative monoid, as in \cite{segal}. It is also related to several other structures that we explore in this section.
	
	\subsection{The Plasmic Nerve is Corepresentable}\label{sec:nerve is corep}
	
	First we describe a system of commutative plasmas, indexed by a functor $\pFin\op\to \ccwumag$, which corepresents the nerve functor. This fits neatly into the nerve-realization framework of \cite{kan-css-functors-nerve} and, to an extent, the following results follow trivially from Proposition \ref{prop:coyoneda for Pn} (see also \cite{nlab:nerve_and_realization}). Nonetheless, we give explicit constructions of all functors and objects involved with the hope that this will be useful to the reader.
	
	\begin{defn}\label{def: power set magma}
		For $n\geq 0$, write $\pow n$ for the power set of $[n]=\{1,2,\ldots,n\}$, where we define $[0]=\varnothing$. To avoid confusion we will write $\emptyset$ for the \textit{element} of $\pow n$ corresponding to the \textit{set} $\varnothing\subseteq [n]$, and retain $\varnothing$ for use as the set with no elements. Define a plasma structure on the power set $\pow n$ by setting
		\[
		X\curlyvee Y=\begin{cases}
			X\cup Y & X\cap Y=\varnothing\\
			\varnothing & X\cap Y\neq \varnothing
		\end{cases}
		\]
		If $\phi\colon \und{n}\to\und{m}$ is a morphism of $\Fin_\ast$, define a function $\pow\phi\colon\pow m\to\pow n$ by setting $\pow\phi(X)=\phi^{-1}(X)$.
	\end{defn}
	
	The following gives an alternative definition of $\pow n$ and is not hard to check.
	
	\begin{prop}\label{prop:coyoneda for Pn}
		Let $\yo\co\colon\pFin\op\to\funmod$ be the coYoneda embedding. Then $\pow n$ is isomorphic, as a plasma, to $\Mag\yo\co\und{n}$
	\end{prop}
	
	\begin{rmk}
		Note  that Definition \ref{def: power set magma} actually gives a \textit{partial monoid} structure on $\pow n$, i.e.~the operation always gives either a unique element or the empty set and it does so associatively. We will implicitly identify the \textit{element} $X\cup Y\in\pow n$ with $\{X\cup Y\}\in \pow{\pow n}$ whenever necessary so that we can think of the given operation as a function $\pow n\times\pow n\to \pow{\pow n}$. It will be especially important to keep this in mind in the proof of Theorem \ref{thm: PM is H}.
	\end{rmk}
	
%	\begin{rmk}
%		The intersection operation $\cap\colon\pow n\times\pow n\to\pow n$ makes $\pow n$ into a subdistributive monoid in $\ccwumag$ so, assuming Conjecture \ref{conj:detplasmas as F1 algebra}, we get that $\HH\pow n\cong \yo\co\und{n}$ is an $\fun$-algebra.
%	\end{rmk}
%	
%	\begin{rmk}
%		Note that the hypermagmas $\pow n$ are not necessarily associative. If we write $\pow 2$ as $\{\emptyset,\{1\},\{2\},\{1,2\}\}=\{0,1,2,3\}$ we see that it has ``addition'' table defined by the following matrix:
%		\[
%		\begin{bmatrix}
%			0 & 1 & 2 & 3\\
%			1 & \varnothing & 3 & \varnothing \\
%			2 & 3 & \varnothing & \varnothing\\
%			3 & \varnothing & \varnothing & \varnothing
%		\end{bmatrix}
%		\] Therefore, $(1\star 2)\star 3=3\star 3=0$ but $1\star (2\star3)=1\star 0=1$. 
%	\end{rmk}

	\begin{rmk}
		In light of Proposition \ref{prop:coyoneda for Pn}, the functor $\mathcal{P}$ in the following lemma is precisely $\Mag\yo\co$, but we feel it is helpful to give an explicit proof of its construction (especially for the purpose of checking that it is fully faithful). 
	\end{rmk}
	
	\begin{lem}\label{lem:power set magma functor}
		For each $\phi\colon \und{n}\to\und{m}$ the function $\pow\phi\colon \pow m\to\pow n$ is a morphism in $\cwumag$. Moreover, the assignments $\und{n}\mapsto \pow n$ and $\phi\mapsto \pow\phi$ assemble into a fully faithful functor $\PP\colon \Fin_\ast\op\to \cwumag$.
	\end{lem}
	
	\begin{proof}
		Let $\phi\colon \und{n}\to\und m$ be a morphism in $\Fin_\ast$ and let $X,Y\in\pow m$. Suppose that $X\cap Y=\varnothing$. Then it is also the case that $\phi^{-1}(X)\cap\phi^{-1}(Y)=\varnothing$, so we have that $\pow\phi(X\curlyvee Y)=\phi^{-1}(X\cup Y)=\phi^{-1}(X)\cup\phi^{-1}(Y)=\pow\phi(X)\curlyvee \pow\phi(Y)$. Now suppose that $X\cap Y\neq\varnothing$. It need not be the case that $\phi^{-1}(X\cap Y)\neq\varnothing$, but nonetheless we have $\pow\phi(X\curlyvee Y)=\phi^{-1}(\varnothing)=\varnothing\subseteq \pow\phi(X)\curlyvee\pow\phi(Y)$. This is all that is necessary to for $\pow\phi$ to be a morphism of $\cwumag$.
		
		To see that this construction is fully faithful, we must check that the function $\pFin(\und{n},\und{m})\to\PL(\pow m,\pow n)$, which takes $\phi$ to $(S\mapsto\phi^{-1}S)$, is a bijection. An inverse of this function is given by taking a plasma map $f\colon \pow m\to \pow n$ to the function $\tilde{f}\colon \und{n}\to\und{m}$ defined by
		\[
		\tilde{f}(i)=\begin{cases}
			k & \text{if}~\exists k,~i\in f(\{k\})\\
			0 & \text{otherwise}
		\end{cases}
		\]
		It is not hard to check that these two constructions are mutually inverse.
	\end{proof}
	
	\begin{defn}
		For $M\in\ccwumag$, write $\PP^M$ for the functor $\Fin_\ast\to\Set_\ast$ given by composing $\PP$ with $\cwumag(-,M)$. This defines a functor $\PP^{(-)}\colon \ccwumag\to \funmod$.
	\end{defn}
	
	\begin{thm}\label{thm: PM is H}
		There is a natural isomorphism $\eta\colon\PP^{(-)}\Rightarrow \hat{H}$ of functors $\ccwumag\to\funmod$.
	\end{thm}
	\begin{proof}
		We begin by describing an isomorphism $\eta_M\colon\PP^M\to\hat{H}M$ for a fixed $M\in\cwumag$. This takes the form of natural isomorphisms $\eta_M\und{n}\colon \PP^M\und{n}\to\hat{H}M\und{n}$ for each $n\in\Fin_\ast$. To show that $\eta_M\und{n}$ is an isomorphism, it suffices to show that a map $f\colon \pow n\to M$ is a morphism of plasmas if and only if $f(S\cup T)\in f(S)\boxplus f(T)$ whenever $S\cap T=\varnothing$.
		
		Suppose that $f\colon \pow n\to M$ is a morphism of hypermagmas, i.e.~$f(S\curlyvee T)\subseteq f(S)\boxplus f(T)$. If $S\cap T=\varnothing$ then $S\curlyvee T=\{S\cup T\}$. Therefore we have $f(S\cup T)=\{f(S\cup T)\}\subseteq f(S)\boxplus f(T)$, i.e.~$f(S\cup T)\in f(S)\boxplus f(T)$. Now suppose that $f(S\cup T)\in f(S)\boxplus f(T)$ whenever $S\cap T=\varnothing$. We need to show that $f(S\curlyvee T)\subseteq f(S)\boxplus f(T)$. Suppose that $S\cap T=\varnothing$, so $S\curlyvee T=\{S\cup T\}$. We have that $f(S\cup T)\in f(S)\boxplus f(T)$ by assumption, so $f(S\curlyvee T)\subseteq f(S)\bp f(T)$. Indeed, this shows that $\PP^M\und{n}=\hat{H}M\und{n}$ as subsets of $M^{\pow n}$. Therefore naturality is immediate. We also need to show naturality with respect to morphisms of $\cwumag$, but this follows immediately from definitions since in both cases morphisms are given by postcomposition. 
	\end{proof}
	
%	\begin{rmk}
%		Note that the plasma $\pow 1$ is $\Mag\fun$. Additionally, the plasma $\pow 2$ is the strict-morphism-classifying hypermagma $\mathcal{E}$ for unital hypermagmas appearing in \cite[Lemma 3.17]{nakamurareyes-mosaics}. It was already shown there that $\PL(\pow 2,M)\cong \HH M_2$ for any plasma $M$, though not stated as such.
%	\end{rmk}
	
	% Recall from \cite[Lemma 3.17]{nakamurareyes-mosaics} that Nakamura and Reyes give a construction of the free unital hypermagma containing three elements, and it is precisely $\pow 2$\So{(This is not free)}. Additionally, $\pow 0=\pow\varnothing$ is isomorphic to $\Mag\fun$.  This pattern of $\pow n$ being the free plasma on $n+1$ elements persists. 
	
	% \begin{prop}
		%     For every $M\in\PL$ there is an isomorphism \[\PL(\pow n,M)\cong \pSet(\und{n},UM)\] where $UM$ is the underlying pointed set of $M$ with $0$ as basepoint.
		% \end{prop}
	
	\subsection{The Segal Condition for Plasma Objects}
	
	Corepresenting the functor $\HH$ as in Section \ref{sec:nerve is corep} allows us to describe a ``Segal condition'' for recognizing its essential image, which we now describe. For simplicity and clarity, we will write $\HH X_1\und{n}$ in this section for the set $(\HH\Mag X)_n$, when $X$ is an $\fun$-module.

	\begin{lem}
		For $M, N\in \ccwumag$, the following is a pullback diagram in $\pSet$.
		\[\begin{tikzcd}
			{\ccwumag(M, N)} && {N^M} \\
			\\
			{\prod\limits_{g\colon \pow 2\to M}\PL(\pow 2, N)} && {\prod\limits_{g\colon \pow 2\to M}N^{\pow 2}}
			\arrow["\prod g^*", from=1-3, to=3-3]
			\arrow["incl.", hook, from=3-1, to=3-3]
			\arrow["incl.", hook, from=1-1, to=1-3]
			\arrow["\prod g^*"', from=1-1, to=3-1]
		\end{tikzcd}\]
		where $g$ runs over all morphisms from $\pow 2$ to $M$ in $\ccwumag$.
	\end{lem}
	\begin{proof}
		A morphism $f\colon M\rightarrow N$ in $\ccwumag$ is a function such that $f(x\bp_M y)\subset f(x)\bp_N f(y)$ for all $x, y\in M$ and $f(0)=0$. These conditions are equivalent to $f\circ g\colon  \pow 2\rightarrow N$ being a morphism in $\ccwumag$ for all $g\colon \pow 2\to M$.
	\end{proof}
	
	By taking $M=\pow n$, we immediately obtain the following:
	
	\begin{cor}
		For $N$ any plasma, the following is a pullback diagram in $\pSet$:
		\[\begin{tikzcd}
			{\HH N_n} && {N^{\pow n}} \\
			\\
			{\prod\limits_{g\colon \pow 2\to \pow n}\hspace{-1.5em}\PL(\pow 2, N)} && {\prod\limits_{g\colon \pow 2\to \pow n}\hspace{-1.5em}N^{\pow 2}}
			\arrow["\prod g^*", from=1-3, to=3-3]
			\arrow["incl.", hook, from=3-1, to=3-3]
			\arrow["incl.", hook, from=1-1, to=1-3]
			\arrow["\prod g^*"', from=1-1, to=3-1]
		\end{tikzcd}\]
	\end{cor}
	
	Note that for any $\fun$-module the $\Mag\dashv \HH$ adjunction induces a unit map $X_n\to \HH X_1\und{n}$. Therefore we have the following diagram, which leads to Corollary \ref{cor:generalized Segal condition}.
	
	% https://q.uiver.app/#q=WzAsNSxbMCwwLCJYX24iXSxbMywxLCJYXntcXHBvdyBufV8xIl0sWzEsMywiXFxwcm9kX2lcXGhhdHtIfVhfMVxcbGFuZ2xlIDJcXHJhbmdsZSJdLFszLDMsIlxccHJvZF97aX1YXntcXHBvdyAyfV8xIl0sWzEsMSwiXFxISCBYXzFcXHVuZHtufSJdLFsxLDMsIlxccHJvZF9pIGlfKiJdLFsyLDMsImluY2wuIiwwLHsic3R5bGUiOnsidGFpbCI6eyJuYW1lIjoiaG9vayIsInNpZGUiOiJ0b3AifX19XSxbMCwxLCJcXHByb2RfUyBcXHJob15YX1MiLDAseyJjdXJ2ZSI6LTJ9XSxbMCwyLCIiLDIseyJjdXJ2ZSI6Mn1dLFswLDRdLFs0LDJdLFs0LDEsIiIsMCx7InN0eWxlIjp7InRhaWwiOnsibmFtZSI6Imhvb2siLCJzaWRlIjoidG9wIn19fV0sWzQsMywiIiwwLHsic3R5bGUiOnsibmFtZSI6ImNvcm5lciJ9fV1d
	\[\begin{tikzcd}
		{X_n} \\
		& {\HH X_1\und{n}} && {X^{\pow n}_1} \\
		\\
		& {\prod_g\hat{H}X_1\langle 2\rangle} && {\prod_{g}X^{\pow 2}_1}
		\arrow["{\prod_g g^*}", from=2-4, to=4-4]
		\arrow["{incl.}", hook, from=4-2, to=4-4]
		\arrow["{\prod_S \rho^X_S}", curve={height=-12pt}, from=1-1, to=2-4]
		\arrow[curve={height=12pt}, from=1-1, to=4-2, "p",swap]
		\arrow[from=1-1, to=2-2]
		\arrow[from=2-2, to=4-2, "\prod g^\ast"]
		\arrow[hook, from=2-2, to=2-4]
		\arrow["\lrcorner"{anchor=center, pos=0.125}, draw=none, from=2-2, to=4-4]
	\end{tikzcd}\]

	% The correspondence in the case $m=2$ gives a natural map
	% $$\prod X_2\twoheadrightarrow \prod\hat{H}X_1\langle 2\rangle.$$
	% Note that this is surjective in general. Hence if the map $$X_2\rightarrow \hat{H}X_1\langle 2\rangle;x\mapsto (\sigma^X_1(x), \sigma^X_2(x), \alpha^X(x))$$
	% is injective, then we obtain an isomorphism $\prod X_2\simeq \prod\hat{H}X_1\langle 2\rangle.$
	
	\begin{cor}\label{cor:generalized Segal condition}
		For a $\fun$-module X, the following are equivalent.
		\begin{enumerate}
			\item $X\cong \HH\Mag X$.
			\item The following is a pullback diagram in $\pSet$ for every $n\geq0$.
			\[\begin{tikzcd}
				{X_n} && {X^{\pow n}_1} \\
				\\
				{\prod_g\hat{H}X_1\langle 2\rangle} && {\prod_{g}X^{\pow 2}_1}
				\arrow["\prod_g g^*", from=1-3, to=3-3]
				\arrow["incl.", hook, from=3-1, to=3-3]
				\arrow["\prod_S \rho^X_S", from=1-1, to=1-3]
				\arrow["p"', from=1-1, to=3-1]
			\end{tikzcd}\]
			where $p$ is the composite of the $\Mag\dashv\HH$ unit for $X$ at $n$ with $\prod_g g^\ast$.
		\end{enumerate}
	\end{cor}
	
	There is a slight modification of the above which avoids using $\HH$ itself in the condition, at the expense of requiring a ``point-set'' level injectivity condition. It is, nonetheless, more evidently similar to the classic Segal condition.
	
	\begin{cor} The equivalent statements of Corollary \ref{cor:generalized Segal condition} are moreover equivalent to the following pair of conditions:
	\begin{enumerate}
	\item The map $X_2\rightarrow X^3_1$ given by \[x\mapsto (\rho_1^{X}(x), \rho_2^{X}(x), \alpha^X(x))\] 
		is injective. 
	\item The following diagram is a pullback diagram in $\pSet$ for every $n\geq0$.
		\[\begin{tikzcd}
			{X_n} && {X^{\pow n}_1} \\
			\\
			{\prod\limits_{\phi\colon \und{n}\to\und{2}}\hspace{-1.2em} X_2} && {\prod\limits_{\phi\colon \und{n}\to\und{2}}\hspace{-1.2em}X^{\pow 2}_1}
			\arrow["\prod_\phi \phi^*", from=1-3, to=3-3]
			\arrow["can.", from=3-1, to=3-3]
			\arrow["\prod_S \rho^X_S", from=1-1, to=1-3]
			\arrow["\prod_\phi \phi^X"', from=1-1, to=3-1]
		\end{tikzcd}\]
		\end{enumerate}
		% where $i$ runs all morphisms in $\textbf{cHMag}^{wu}(P_2, P_n)$.
		% The following diagram might look better:
		% \[\begin{tikzcd}
			% 	{X_n} && {\prod _{\langle n\rangle\rightarrow\langle 1\rangle}X_1} \\
			% 	\\
			% 	{\prod_{\langle n\rangle\rightarrow\langle 2\rangle} X_2} && {\prod_{\langle n\rangle\rightarrow\langle 2\rangle}\prod_{\langle 2\rangle\rightarrow\langle 1\rangle}X_1}
			% 	\arrow["\circ_*", from=1-3, to=3-3]
			% 	\arrow["can.", from=3-1, to=3-3]
			% 	\arrow["can.", from=1-1, to=1-3]
			% 	\arrow["can."', from=1-1, to=3-1]
			% \end{tikzcd}\]
	\end{cor}
	%\begin{proof}If $X=\hat{H}X_1$, the previous lemma implies that the square above is a pullback diagram. Assume that the square is a pullback diagram.
	%\[\begin{tikzcd}
		%	S \\
		%{\prod_iX_2} && {\prod_icHMag(P_2, X_1)} && {\prod_{i}X^{P_2}_1}
		%	\arrow["\prod_i i_*", from=2-5, to=4-5]
		%	\arrow["incl.", hook, from=4-3, to=4-5]
		%\arrow["\prod_S\rho^X_S"',from=2-3, to=2-5]
		%\arrow["\prod_i ", from=2-3, to=4-3]
		%\arrow["can.", two heads, from=4-1, to=4-3]
		%\arrow["f", from=1-1, to=4-1]
		%\arrow["\prod_i i^X", from=2-3, to=4-1]
		%\arrow["g", from=1-1, to=2-5]
		%\arrow["{\exists!}"{description}, dashed, from=1-1, to=2-3]
		%\end{tikzcd}\]
		%\end{proof}
		
		% \jb{Can we use this to show that if $M$ is a commutative monoid then $\hat{H}M$ satisfies the ``usual'' Segal condition in terms of the isomorphism $\hat{H}M\und{n}\to\prod_n \hat{H}M\und{1}$? I'm having trouble figuring out what $i^X$ and $i_\ast$ are.}
		
		% \So{Fixing $i:\langle n\rangle\rightarrow \langle 2\rangle$ (we are identifying $\textbf{Set}_*(\langle n\rangle, \langle 2\rangle)$ and $\textbf{cHMag}(P_2, P_n)$), we mean by $i^X$ the image of $i$ under $X$. $i_*$ stands for the map $X_1^{P_n}\rightarrow X_1^{P_2}$ given by the composition $f\mapsto f\circ i^{-1}$. Here $i^{-1}: P_2\rightarrow P_n$ is the hypermagma morphism corresponding to $i$. }
		
		One might like to use the characterization of Corollary \ref{cor:generalized Segal condition} to give a definition of ``internal plasma object'' of a category $\cC$ with finite limits and finite coproducts in terms of pointed functors $\pFin\to\cC$. The most immediate obstruction is that the functor $\HH$ must be defined. One can slightly generalize the $\Mag\dashv \HH$ adjunction to a ``2-coskeletonization,'' i.e.~precomposition and right Kan extension, adjunction between $\Fun_\ast(\pFin,\cC)$ and $\Fun_\ast(\pFin^{\und{2}},\cC)$, where $\pFin^{\und{2}}$ is the full subcategory of $\pFin$ on $\und{0}$, $\und{1}$ and $\und{2}$. In the case that $\cC=\pSet$, this adjunction is very close to the $\Mag\dashv\HH$ adjunction, with the caveat that arbitrary pointed functors $\pFin^{\und{2}}\to\pSet$ still allow for ``redundancy'' of witnesses. We suspect that the ``correct'' generalization involves a modified $\HH$ functor which goes from (a certain class of commutative magmas in) $Span(\cC)$ to $\Fun_\ast(\pFin,\cC)$. We draw some connections between our functor and spans of pointed sets in Section \ref{sec:nerves not 2-segal} but leave a full investigation of this relationship to future work.

		\subsection{The Plasmic Nerve and Segal's Constructions}\label{sec:nerve and Segal constructions}

		In \cite{segal} a functor $\Ab\to\Fun_\ast(\pFin,\TTop_\ast)$ is described which evidently factors through $\funmod$ and whose domain extends to $\CMon$. We recall that functor here.
		
		\begin{defn}[\cite{segal}]\label{defn:SegalHfunctor}
			Let $A$ be a commutative monoid. Define a functor $H\colon \CMon\to \funmod$ by defining $HA_n=A^n$ and, for a map $\phi\colon \und{n}\to\und{m}$ in $\pFin$, define $A^n\to A^m$ by the formula:
			\[
			(a_1,a_2,\ldots,a_n)\mapsto \left(\sum_{i\in\phi^{-1}(j)}a_i\right)_{1\leq j\leq m}
			\]
		\end{defn}
		
		\begin{rmk}
			We often call this the \textit{Eilenberg-MacLane functor}. Note that there is an inclusion $\funmod\hookrightarrow\Fun_\ast(\pFin,\TTop_\ast)$ of $\fun$-modules as the discrete objects in the category of what Segal called \textit{$\Gamma$-spaces}. The latter category, with a suitable Quillen model structure, is Quillen equivalent to the category of connective spectra (see, for instance, \cite{localKtheory}). Under this equivalence, the object $HA$ corresponds to the classical Eilenberg-MacLane spectrum of the group completion $K_0(A)$.
		\end{rmk}
		
		\begin{thm}\label{thm:H hat generalizes H}
			If $M\in \ccwumag$ is a commutative monoid, then $HM$ is naturally isomorphic to $\hat{H}M$. 
		\end{thm}
		
		\begin{proof}
			Note that if $M$ is a commutative monoid then there is an isomorphism $\hat{H}M_n\to M^n$ given by projecting onto the coordinates corresponding to the one element subsets of $[n]$. This is an isomorphism because for each subset of cardinality greater than one, say $I=\{i_1,\ldots,i_k\}\subseteq [n]$, there is a unique element which satisfies the necessary conditions to be the coordinate of $\hat{H}M_n$ corresponding to $I$. Specifically, it must be the sum $\sum_{i\in I}a_i$. 
			
			Now suppose we are given a morphism $\phi\colon \und{n}\to\und{m}$ in $\Fin_\ast$ and consider the following diagram:
			\[
			\begin{tikzcd}
				\hat{H}M_n\ar[r,"\cong"]\ar[d,"\hat{H}M\phi", swap] & M^n\ar[d,"HM\phi"]\\
				\hat{H}M_m\ar[r,"\cong"] & M^m
			\end{tikzcd}
			\]
			Let $\vec{a}=(a_T)_{T\subseteq[n]}\in\hat{H}M_n$. For each $S\subseteq[m]$ we have $\hat{H}M\phi(\vec{a})_{S}=a_{\phi^{-1}(S)}$. In particular, because $M$ is a commutative monoid, for each $j\in[m]$ we have $\hat{H}M\phi(\vec{a})_{\{j\}}=\sum_{i\in\phi^{-1}(j)}a_i$, and therefore the diagram commutes. 
		\end{proof}
		
		In \cite{segal-config_spaces}, Segal constructs a model for the configuration space of $n$ points in a pointed topological space $X$, which he denotes $C_n(X)$. He wishes to think of this space as a topological monoid so that he can ``deloop'' it. But the desired monoid structure only makes it a partial monoid. Nonetheless, Segal shows that one can construct a delooping, or classifying space, of a partial monoid. We recall (the discrete version of) this definition here so that we can see how our $\HH$ functor recovers it.
		
		\begin{defn}
			Let $M$ be a partial monoid. Its classifying space is the simplicial set $BM_\bullet$ with:
			\begin{enumerate}
				\item Set of $n$ simplices defined to be: \[BM_n=\{(m_1,m_2,\ldots,m_n)\in M^n:\text{the sum~}m_1+m_2+\cdots+m_n\text{~is defined in}~M\}\]
				\item Face maps $d_i\colon BM_n\to BM_{n-1}$ defined by:
				\[
				d_i(m_1,\ldots,m_n)=\begin{cases}
					(m_2,\ldots,m_n) & i=0\\
					(m_1,\ldots,m_i+m_{i+1},\ldots,m_n) & 0<i<n\\
					(m_1,\ldots,m_{n-1}) & i=n
				\end{cases}
				\]
				\item Degeneracy maps $s_i\colon BM_{n}\to BM_{n+1}$ defined by:
				\[
				s_i(m_1,\ldots,m_n)=(m_1,\ldots,m_i,0,m_{i+1},\ldots,m_n)
				\]
			\end{enumerate}
		\end{defn}

		Given a commutative partial monoid $M$, realized as a plasma, we may take its associated $\fun$-module $\HH M$. While $\HH M$ is not a simplicial set, it has an ``underlying'' simplicial set $B\HH M_\bullet$ as described in Definition \ref{defn:delooping}. We will use the constructions of Appendix \ref{sec:delooping appendix} in the following.
		
		\begin{prop}\label{prop:nerve give partial monoid nerve}
			For any commutative partial monoid $M$, there is an equality of simplicial sets $B\HH M_\bullet= BM_\bullet$.
		\end{prop}
		
		\begin{proof}
			Because $M$ is deterministic and associative, the elements of $\HH M_n=B\HH M_n$ are precisely $n$-tuples of elements of $M$ whose collective sum exists. Therefore the two simplicial sets have equal sets of $n$-simplices. The description of the face maps and degeneracy maps of $B\HH M_\bullet$, given in Appendix \ref{sec:delooping appendix}, is evidently identical to the same data in $BM_\bullet$.
		\end{proof}
		
		\begin{rmk}
			Of course if $M$ is a commutative (not-partial) monoid then $BM_\bullet$ is precisely the usual classifying quasicategory of $M$, hence the notation. Therefore we also have $BM_\bullet\cong BHM_\bullet$, which is well known.
		\end{rmk}
		
		\subsection{Plasmic Nerves are Not Always 2-Segal}\label{sec:nerves not 2-segal}
		
		In this section, we contrast our constructions with the work of \cite{contreras-mehta-stern}. Recall that Contreras, Mehta and Stern consider commutative pseudomonoids in the bicategory of spans of sets, which we write as $\CMon^{ps}(Span(\Set))$. They give a functor from this category to the category of functors $F\colon\pFin\to\Set$ (not necessarily pointed) and characterize its image as the full subcategory of functors such that $F\circ \beta$ is 2-Segal, in the sense of \cite{dyckerhoff-kapranov_higher_segal}. Therefore, given an $\fun$-module $X$, one may ask if the composite $\Delta\op\xrightarrow{\beta} \pFin\xrightarrow{X}\pSet\to\Set$ is in this image, where the last functor forgets the basepoint. 
		
		\begin{defn}
			A simplicial set $X\colon \Delta\op\to\Set$ is called 2-Segal if the following two diagrams are pullback diagrams for all $0<i<n$, where $d_i$ are the standard face maps.
			% https://q.uiver.app/#q=WzAsOCxbMCwwLCJYX3tuKzF9Il0sWzEsMCwiWF9uIl0sWzEsMSwiWF97bi0xfSJdLFswLDEsIlhfbiJdLFszLDAsIlhfe24rMX0iXSxbNCwwLCJYX24iXSxbNCwxLCJYX3tuLTF9Il0sWzMsMSwiWF97bn0iXSxbMCwxLCJkXzAiXSxbMSwyLCJkX2kiXSxbMCwzLCJkX3tpKzF9IiwyXSxbMywyLCJkXzAiLDJdLFswLDIsIiIsMSx7InN0eWxlIjp7Im5hbWUiOiJjb3JuZXIifX1dLFs0LDUsImRfe24rMX0iXSxbNSw2LCJkX2kiXSxbNCw3LCJkX2kiLDJdLFs3LDYsImRfbiIsMl0sWzQsNiwiIiwxLHsic3R5bGUiOnsibmFtZSI6ImNvcm5lciJ9fV1d
			\[\begin{tikzcd}
				{X_{n+1}} & {X_n} && {X_{n+1}} & {X_n} \\
				{X_n} & {X_{n-1}} && {X_{n}} & {X_{n-1}}
				\arrow["{d_0}", from=1-1, to=1-2]
				\arrow["{d_i}", from=1-2, to=2-2]
				\arrow["{d_{i+1}}"', from=1-1, to=2-1]
				\arrow["{d_0}"', from=2-1, to=2-2]
				\arrow["\lrcorner"{anchor=center, pos=0.125}, draw=none, from=1-1, to=2-2]
				\arrow["{d_{n+1}}", from=1-4, to=1-5]
				\arrow["{d_i}", from=1-5, to=2-5]
				\arrow["{d_i}"', from=1-4, to=2-4]
				\arrow["{d_n}"', from=2-4, to=2-5]
				\arrow["\lrcorner"{anchor=center, pos=0.125}, draw=none, from=1-4, to=2-5]
			\end{tikzcd}\]
		\end{defn}
		
		The relevant result of \cite{contreras-mehta-stern} can be summarized as follows:
		
		\begin{thm}[\cite{contreras-mehta-stern}]
			There is a fully faithful functor $\CMon^{ps}(Span(\Set))\to\Fun(\pFin,\Set)$ whose essential image is characterized as the full subcategory of functors $F\colon\pFin\to\Set$ such that $F\circ \beta$ is 2-Segal. 
		\end{thm}
		
%		Note that a set $X$ with a weakly unital, commutative hyperoperation $\star\colon X\times X\to\pow X$ can be realized as a structure in $Span(\Set)$. Specifically, a span $X\times X\xleftarrow{} Z\to X$, with $Z$ being $\hat{H}X_2$, and a span $\{0\}\leftarrow \{0\}\to X$, satisfying certain compatibility conditions. Therefore one might hope that given an actual commutative (pseudo)monoid object $M$ in $Span(\Set)$, Contreras, Mehta and Stern's functor might recover $\hat{H}M$. This, however, is not generally the case.

Note that a pseudomonoid $X\in\CMon^{ps}(Span(\Set))$ has, in particular, a hyperoperation determined by the monoid structure span $X\times X\leftarrow Z\to X$ (for some set $Z$). Therefore both their work and ours yield methods for producing $\fun$-modules from commutative hyperstructures. The following proposition asserts that these two methods cannot entirely overlap.
		
		\begin{prop}
			It is not that case that $B\hat{H}M$, for $M$ a commutative plasma, is necessarily a 2-Segal simplicial set.
		\end{prop}
		
		\begin{proof}
			We exhibit a counterexample. We will write $d_k$ for $\hat{H}\mathbb{K}(\beta\delta_k)$, as described in Appendix \ref{sec:delooping appendix}. Let $\K=\{0,1\}$ be the Krasner hyperfield with hyperoperation $0\boxplus 0=0$, $1\boxplus 0=1=0\boxplus 1$ and $1\boxplus 1=\{0,1\}$. It suffices to show that the following is not a pullback diagram in $\Set$:
			% https://q.uiver.app/#q=WzAsNCxbMCwwLCJcXGhhdHtIfVxcS18zIl0sWzEsMCwiXFxoYXR7SH1cXEtfMiJdLFsxLDEsIlxcaGF0e0h9XFxLXzE9XFxLIl0sWzAsMSwiXFxoYXR7SH1cXEtfMiJdLFswLDMsImRfMiIsMl0sWzAsMSwiZF8wIl0sWzEsMiwiZF8xIl0sWzMsMiwiZF8wIiwyXV0=
			\[\begin{tikzcd}
				{B\hat{H}\K_3} & {B\hat{H}\K_2} \\
				{B\hat{H}\K_2} & {B\hat{H}\K_1=\K}
				\arrow["{d_2}"', from=1-1, to=2-1]
				\arrow["{d_0}", from=1-1, to=1-2]
				\arrow["{d_1}", from=1-2, to=2-2]
				\arrow["{d_0}"', from=2-1, to=2-2]
			\end{tikzcd}\]
			We explicitly describe the maps in the above diagram. Suppose we have an element $\bar{a}=(a_1,a_2,a_3,a_{12},a_{13},a_{23}, a_{123})$ in $B\hat{H}\K_3$. Then $d_0(\bar{a})=(a_2,a_3,a_{23})$ and $d_2(\bar{a})=(a_1,a_{23},a_{123})$. Given $\bar{b}=(b_1,b_2,b_{12})$ in $B\hat{H}\K_2$, we have $d_0(\bar{b})=b_2$ and $d_1(\bar{b})=b_{12}$. It is clear that this diagram commutes. We now compute the pullback of its bottom right cospan. 
			
			Recall that $\hat{H}\K_2=\{(0,0,0), (1,0,1), (0,1,1),(1,1,1),(1,1,0)\}$. The pullback of interest is the collection of pairs $((a_1,a_2,a_{12}),(b_1,b_2,b_{12}))$ such that $a_2=b_{12}$. Therefore it is the set (where we use semicolons to separate distinct elements):
			\[
			\begin{Bmatrix}
				\left((0,0,0),(0,0,0)\right);&((0,0,0),(1,1,0));\\ ((1,0,1),(0,0,0));& ((1,0,1),(1,1,0));\\
				((0,1,1),(1,0,1));& ((0,1,1),(0,1,1));\\
				((0,1,1),(1,1,1));& ((1,1,1),(1,0,1));\\
				((1,1,1),(0,1,1));& ((1,1,1),(1,1,1));\\
				((1,1,0),(1,0,1));&((1,1,0),(0,1,1));\\
				((1,1,0),(1,1,1))&
			\end{Bmatrix}\]
			which has thirteen elements.
			
			On the other hand, we can compute $B\hat{H}\K_3$ to be the set:
			\[
			\begin{Bmatrix}
				(0,0,0,0,0,0,0)&
				(0,0,1,0,1,1,1)\\
				(0,1,0,1,0,1,1)&
				(1,0,0,1,1,0,1)\\
				(0,1,1,1,1,0,0)&
				(0,1,1,1,1,1,1)\\
				(1,0,1,1,0,1,0)&
				(1,0,1,1,1,1,1)\\
				(1,1,0,0,1,1,0)&
				(1,1,0,1,1,1,1)\\
				(1,1,1,0,0,0,1)&
				(1,1,1,1,0,0,1)\\
				(1,1,1,0,1,0,1)&
				(1,1,1,0,0,1,1)\\
				(1,1,1,1,1,0,1)&
				(1,1,1,1,0,1,1)\\
				(1,1,1,0,1,1,1)&
				(1,1,1,1,1,1,0)\\
				(1,1,1,1,1,1,1)
			\end{Bmatrix}
			\] which has nineteen elements. Therefore the two sets cannot be isomorphic.
		\end{proof}
		
		\begin{rmk}
			In general, the cardinality of the sets $\HH M_n$ grows very fast with $n$. For instance $\HH\K_4$ has 137 elements. We leave further computations to the interested reader. 
		\end{rmk}

			One can check that $B\fun$ is 2-Segal (as the nerve of a certain partial category). Therefore it is not the case that there is no overlap between our functor and the functor of \cite{contreras-mehta-stern}. Indeed, so long as we restrict ourselves to partial commutative monoids, our Eilenberg-MacLane functor lands entirely within the image of the functor constructed in \cite{contreras-mehta-stern}.
		
		\begin{prop}
			If $M\in\ccwumag$ is a partial commutative monoid then $B\hat{H}M$ is $2$-Segal.
		\end{prop}
		
		\begin{proof}
			Let $(M,+,0)$ be a commutative partial monoid, thought of as an object of $\ccwumag$. Note that because $M$ is associative and deterministic, each set $\hat{H}M\und{n}$ is isomorphic to the subset of $M^n$ composed of $n$-tuples $(a_1,\ldots,a_n)$ whose total sum exists in $M$. Examination of $\hat{H}M$ on morphisms shows that the face maps and degeneracy maps are the standard ones where outer face maps discard outer elements, inner face maps add adjacent elements, and degeneracy maps insert the unit element. Therefore by \cite{bergner_et_al-2Segal_Waldhausen} it is 2-Segal.
		\end{proof}

		\begin{appendices}
\section{The Simplicial Set Underlying an $\fun$-module}\label{sec:delooping appendix}
There is a well-known functor $\Delta\op\to\pFin$, sometimes referred to as $\mathrm{Cut}$, which via precomposition associates a pointed simplicial set to any $\fun$-module. We review this functor and some of its properties in this appendix. Throughout this appendix, $[n]$ denotes the ordered set $\{0,1,\ldots,n\}$, regarded as an object of $\Delta$.

\begin{defn}\label{defn:delooping}
Let $\beta\colon\Delta\op\to\pFin$ be the functor with $\beta([n])=\und{n}$. For a morphism $\phi\colon[m]\to[n]$ in $\Delta$, write $\beta(\phi)\colon\und{n}\to\und{m}$ for its image under $\beta$, with the opposite understood. Set $\beta(\phi)(0)=0$. For each $1\leq i\leq n$, set
\[
K_i=\{j\in[m]:i\leq\phi(j)\}
\]
and define
\[
\beta(\phi)(i)=
\begin{cases}
\min(K_i) & K_i\neq\varnothing\text{ and }0\notin K_i,\\
0 & K_i=\varnothing\text{ or }0\in K_i.
\end{cases}
\]
\end{defn}

\begin{rmk}
The definition above identifies $\beta$ with the pointed simplicial set $\Delta^1/\partial\Delta^1$, corestricted to $\pFin$. Indeed, its non-basepoint $n$-simplices are the nonconstant order-preserving maps $[n]\to[1]$. Identify $i\in\{1,\ldots,n\}$ with the map
\[
\chi_i(j)=
\begin{cases}
0 & j<i,\\
1 & j\geq i.
\end{cases}
\]
The two constant maps represent the basepoint. Precomposition with $\phi\colon[m]\to[n]$ sends $\chi_i$ to the basepoint if $K_i=\varnothing$ or $0\in K_i$, and otherwise to $\chi_{\min(K_i)}$. This recovers the formula above and verifies its functoriality.

Thus $\beta$ has a unique $0$-simplex, a single nondegenerate $1$-simplex, and no nondegenerate simplices in higher degrees. It is sometimes denoted $S^1$, since its geometric realization is homeomorphic to the circle. It is not a Kan complex, nor even a quasicategory: the inner horn $\Lambda^2_1\to\Delta^1/\partial\Delta^1$ whose two consecutive edges are both the nondegenerate $1$-simplex has no filler.
\end{rmk}

For the reader's convenience, we explicitly describe the images of some of the coface and codegeneracy maps in $\Delta$ under $\beta$ in Tables~\ref{table:facedegen12} and~\ref{degen32}. These illustrate the following result.

\begin{prop}
For $n\geq0$ and $0\leq k\leq n$, let $\sigma_k^n\colon[n+1]\to[n]$ denote the unique morphism in $\Delta$ with $(\sigma_k^n)^{-1}(k)=\{k,k+1\}$ and $|(\sigma_k^n)^{-1}(i)|=1$ for every $i\neq k$. For $n\geq1$ and $0\leq k\leq n$, let $\delta_k^n\colon[n-1]\to[n]$ be the unique injective morphism in $\Delta$ whose image omits $k$. Then:
\begin{enumerate}
\item For $n\geq1$ and $0\leq k<n$, the map $\beta(\delta_k^n)\colon\und{n}\to\und{n-1}$ is given by
\[
\beta(\delta_k^n)(i)=
\begin{cases}
i & i\leq k,\\
i-1 & i>k.
\end{cases}
\]
\item For $n\geq1$, the map $\beta(\delta_n^n)\colon\und{n}\to\und{n-1}$ is given by
\[
\beta(\delta_n^n)(i)=
\begin{cases}
i & i\neq n,\\
0 & i=n.
\end{cases}
\]
\item For $n\geq0$ and $0\leq k\leq n$, the map $\beta(\sigma_k^n)\colon\und{n}\to\und{n+1}$ is given by
\[
\beta(\sigma_k^n)(i)=
\begin{cases}
i & i\leq k,\\
i+1 & i>k.
\end{cases}
\]
\end{enumerate}
\end{prop}

\begin{proof}
All of these maps preserve $0$, so it suffices to consider $1\leq i\leq n$.

First suppose $0\leq k<n$. Since
\[
\delta_k^n(j)=
\begin{cases}
j & j<k,\\
j+1 & j\geq k,
\end{cases}
\]
the set $K_i$ in Definition~\ref{defn:delooping} is $\{i,\ldots,n-1\}$ if $i\leq k$, and $\{i-1,\ldots,n-1\}$ if $i>k$. In the first case its minimum is $i$. In the second case its minimum is $i-1$; when $i=1$ this set contains $0$, so the definition also gives $0=i-1$. This proves the first formula.

For $k=n$, we have $\delta_n^n(j)=j$ for every $j\in[n-1]$. Thus $K_i=\{i,\ldots,n-1\}$ for $1\leq i<n$, whereas $K_n=\varnothing$. This proves the second formula.

Finally,
\[
\sigma_k^n(j)=
\begin{cases}
j & j\leq k,\\
j-1 & j>k.
\end{cases}
\]
Consequently, $K_i=\{i,\ldots,n+1\}$ for $1\leq i\leq k$, and $K_i=\{i+1,\ldots,n+1\}$ for $i>k$. Taking minima gives the third formula.
\end{proof}

In the notation introduced earlier, these formulas give
\[
\beta(\delta_0^2)=\rho_2,\qquad
\beta(\delta_1^2)=\alpha,\qquad
\beta(\delta_2^2)=\rho_1,
\]
and
\[
\beta(\sigma_0^1)=i_2,\qquad
\beta(\sigma_1^1)=i_1.
\]

\begin{defn}
Write $B\colon\funmod\to\sSet_\ast$ for the functor given by precomposition with $\beta\colon\Delta\op\to\pFin$. Explicitly, for an $\fun$-module $X$, we have $BX=X\circ\beta$, so that $(BX)_n=X(\und{n})$, with face and degeneracy maps
\[
d_k=X\bigl(\beta(\delta_k^n)\bigr)\colon(BX)_n\to(BX)_{n-1},
\qquad
s_k=X\bigl(\beta(\sigma_k^n)\bigr)\colon(BX)_n\to(BX)_{n+1}.
\]
\end{defn}

\begin{table}[p]
    \centering
    {\begin{tabular}{|c|c|c|}
            \hline
            $f$ & $[1]\to[2]$ & $\langle 2\rangle\to\langle 1\rangle$\\
            \hline
            $\delta_0^2$& \begin{tikzcd}[column sep=small,row sep=tiny]
                0 && 0 \\
                1 && 1 \\
                && 2
                \arrow[from=2-1, to=3-3]
                \arrow[from=1-1, to=2-3]
            \end{tikzcd}&\begin{tikzcd}[column sep=small,row sep=tiny]
                0 && 0 \\
                1 && 1 \\
                2
                \arrow[from=1-1, to=1-3]
                \arrow[from=2-1, to=1-3]
                \arrow[from=3-1, to=2-3]
            \end{tikzcd}\\
            \hline
            $\delta_1^2$ &
            \begin{tikzcd}[column sep=small,row sep=tiny]
                0 && 0 \\
                1 && 1 \\
                && 2
                \arrow[from=1-1, to=1-3]
                \arrow[from=2-1, to=3-3]
            \end{tikzcd}& \begin{tikzcd}[column sep=small,row sep=tiny]
                0 && 0 \\
                1 && 1 \\
                2
                \arrow[from=1-1, to=1-3]
                \arrow[from=2-1, to=2-3]
                \arrow[from=3-1, to=2-3]
            \end{tikzcd}\\
            \hline
            $\delta_2^2$ & \begin{tikzcd}[column sep=small,row sep=tiny]
                0 && 0 \\
                1 & &1 \\
                & &2
                \arrow[from=1-1, to=1-3]
                \arrow[from=2-1, to=2-3]
            \end{tikzcd} & \begin{tikzcd}[column sep=small,row sep=tiny]
                0 && 0 \\
                1 && 1 \\
                2
                \arrow[from=1-1, to=1-3]
                \arrow[from=2-1, to=2-3]
                \arrow[from=3-1, to=1-3]
            \end{tikzcd}\\
            \hline
            \rule{0pt}{19pt} &\raisebox{4.5pt}{$[2]\to[1]$} & \raisebox{4.5pt}{$\langle 1\rangle\to\langle 2\rangle$}\\
            \hline
            $\sigma_0^1$ &\begin{tikzcd}[column sep=small,row sep=tiny]
                0 && 0 \\
                1 && 1 \\
                2
                \arrow[from=1-1, to=1-3]
                \arrow[from=2-1, to=1-3]
                \arrow[from=3-1, to=2-3]
            \end{tikzcd}  & \begin{tikzcd}[column sep=small,row sep=tiny]
                0 && 0 \\
                1 && 1 \\
                && 2
                \arrow[from=1-1, to=1-3]
                \arrow[from=2-1, to=3-3]
            \end{tikzcd}\\
            \hline
            $\sigma_1^1$&\begin{tikzcd}[column sep=small,row sep=tiny]
                0 && 0 \\
                1 && 1 \\
                2
                \arrow[from=1-1, to=1-3]
                \arrow[from=3-1, to=2-3]
                \arrow[from=2-1, to=2-3]
            \end{tikzcd}&\begin{tikzcd}[column sep=small,row sep=tiny]
                0 && 0 \\
                1 && 1 \\
                && 2
                \arrow[from=1-1, to=1-3]
                \arrow[from=2-1, to=2-3]
            \end{tikzcd}\\
            \hline
    \end{tabular}}
    \quad
    \begin{tabular}{|c|c|c|}
        \hline
        $f$ & $[2]\to[3]$ & $\langle 3\rangle\to\langle 2\rangle$\\
        \hline
        $\delta_0^3$&\begin{tikzcd}[column sep=small,row sep=tiny]
            0 && 0 \\
            1 && 1 \\
            2 && 2 \\
            && 3
            \arrow[from=1-1, to=2-3]
            \arrow[from=3-1, to=4-3]
            \arrow[from=2-1, to=3-3]
        \end{tikzcd}&\begin{tikzcd}[column sep=small,row sep=tiny]
            0 && 0 \\
            1 && 1 \\
            2 && 2 \\
            3
            \arrow[from=1-1, to=1-3]
            \arrow[from=2-1, to=1-3]
            \arrow[from=4-1, to=3-3]
            \arrow[from=3-1, to=2-3]
        \end{tikzcd}\\
        \hline
        $\delta_1^3$&\begin{tikzcd}[column sep=small,row sep=tiny]
            0 && 0 \\
            1 && 1 \\
            2 && 2 \\
            && 3
            \arrow[from=1-1, to=1-3]
            \arrow[from=3-1, to=4-3]
            \arrow[from=2-1, to=3-3]
        \end{tikzcd}& \begin{tikzcd}[column sep=small,row sep=tiny]
            0 && 0 \\
            1 && 1 \\
            2 && 2 \\
            3
            \arrow[from=1-1, to=1-3]
            \arrow[from=2-1, to=2-3]
            \arrow[from=4-1, to=3-3]
            \arrow[from=3-1, to=2-3]
        \end{tikzcd}\\
        \hline
        $\delta_2^3$&\begin{tikzcd}[column sep=small,row sep=tiny]
            0 && 0 \\
            1 && 1 \\
            2 && 2 \\
            && 3
            \arrow[from=1-1, to=1-3]
            \arrow[from=2-1, to=2-3]
            \arrow[from=3-1, to=4-3]
        \end{tikzcd}&\begin{tikzcd}[column sep=small,row sep=tiny]
            0 && 0 \\
            1 && 1 \\
            2 && 2 \\
            3
            \arrow[from=1-1, to=1-3]
            \arrow[from=2-1, to=2-3]
            \arrow[from=3-1, to=3-3]
            \arrow[from=4-1, to=3-3]
        \end{tikzcd}\\
        \hline
        $\delta_3^3$&\begin{tikzcd}[column sep=small,row sep=tiny]
            0 && 0 \\
            1 && 1 \\
            2 && 2 \\
            && 3
            \arrow[from=1-1, to=1-3]
            \arrow[from=2-1, to=2-3]
            \arrow[from=3-1, to=3-3]
        \end{tikzcd} &\begin{tikzcd}[column sep=small,row sep=tiny]
            0 && 0 \\
            1 && 1 \\
            2 && 2 \\
            3
            \arrow[from=1-1, to=1-3]
            \arrow[from=2-1, to=2-3]
            \arrow[from=3-1, to=3-3]
            \arrow[from=4-1, to=1-3]
        \end{tikzcd}\\
        \hline
    \end{tabular}
    \caption{The left table shows the effect of $\beta$ on the coface and codegeneracy maps between $[1]$ and $[2]$. The right table shows the effect of $\beta$ on the coface maps $[2]\to[3]$.}
    \label{table:facedegen12}
\end{table}

\begin{table}[htbp]
    \centering
        \begin{tabular}{|c|c|c|}
            \hline
            $f$ & $[3]\to[2]$ & $\langle 2\rangle\to\langle 3\rangle$\\
            \hline
            $\sigma_0^2$&\begin{tikzcd}[ column sep=small,row sep=tiny]
                0 && 0 \\
                1 && 1 \\
                2 && 2 \\
                3
                \arrow[from=1-1, to=1-3]
                \arrow[from=2-1, to=1-3]
                \arrow[from=3-1, to=2-3]
                \arrow[from=4-1, to=3-3]
            \end{tikzcd}&\begin{tikzcd}[column sep=small,row sep=tiny]
                0 && 0 \\
                1 && 1 \\
                2 && 2 \\
                && 3
                \arrow[from=1-1, to=1-3]
                \arrow[from=2-1, to=3-3]
                \arrow[from=3-1, to=4-3]
            \end{tikzcd}\\
            \hline
            $\sigma_1^2$&\begin{tikzcd}[ column sep=small,row sep=tiny]
                0 && 0 \\
                1 && 1 \\
                2 && 2 \\
                3
                \arrow[from=1-1, to=1-3]
                \arrow[from=2-1, to=2-3]
                \arrow[from=3-1, to=2-3]
                \arrow[from=4-1, to=3-3]
            \end{tikzcd}&\begin{tikzcd}[column sep=small,row sep=tiny]
                0 && 0 \\
                1 && 1 \\
                2 && 2 \\
                && 3
                \arrow[from=1-1, to=1-3]
                \arrow[from=2-1, to=2-3]
                \arrow[from=3-1, to=4-3]
            \end{tikzcd}\\
            \hline
            $\sigma_2^2$&\begin{tikzcd}[ column sep=small,row sep=tiny]
                0 && 0 \\
                1 && 1 \\
                2 && 2 \\
                3
                \arrow[from=1-1, to=1-3]
                \arrow[from=2-1, to=2-3]
                \arrow[from=3-1, to=3-3]
                \arrow[from=4-1, to=3-3]
            \end{tikzcd}&\begin{tikzcd}[column sep=small,row sep=tiny]
                0 && 0 \\
                1 && 1 \\
                2 && 2 \\
                && 3
                \arrow[from=1-1, to=1-3]
                \arrow[from=2-1, to=2-3]
                \arrow[from=3-1, to=3-3]
            \end{tikzcd}\\
            \hline
        \end{tabular}
        \caption{The images under $\beta$ of the codegeneracy maps $[3]\to[2]$.}
        \label{degen32}

\end{table}

\begin{example}
    Consider $B\fun$, which is the composition $\Delta\op\xrightarrow{\beta}\pFin\hookrightarrow\pSet$. This is the pointed simplicial set $\Delta^1/\partial\Delta^1$ described above, with $(B\fun)_n=\fun(\und{n})=\und{n}$. Let $f$ denote its unique nondegenerate $1$-simplex, and write $\operatorname{id}_0=s_0(0)$ for the degenerate edge at its unique vertex. The three $2$-simplices are $s_0(f)$, $s_1(f)$, and $s_0s_0(0)$, depicted in that order below:

    \[\begin{tikzcd}
        & 0 \\
        0 && 0 \\
        & 0 & {} \\
        0 && 0 \\
        & 0 \\
        0 && 0
        \arrow["f"', from=2-1, to=2-3]
        \arrow["{\operatorname{id}_0}", from=2-1, to=1-2]
        \arrow["f", from=1-2, to=2-3]
        \arrow["f", from=4-1, to=3-2]
        \arrow["{\operatorname{id}_0}", from=3-2, to=4-3]
        \arrow["f"', from=4-1, to=4-3]
        \arrow["{\operatorname{id}_0}"', from=6-1, to=6-3]
        \arrow["{\operatorname{id}_0}", from=6-1, to=5-2]
        \arrow["{\operatorname{id}_0}", from=5-2, to=6-3]
    \end{tikzcd}\]
    In other words, $B\fun$ ``knows'' how to compose the morphism $f$ with the identity, and knows how to compose the identity with the identity, but doesn't know how to compose $f$ with itself. It is in this sense that $B\fun$ is a ``delooping'' of $\fun$. Higher simplices enforce associativity (when the composition is well defined). A complete computation of the ``spectrum'' of iterated deloopings of $\fun$ is made in \cite{beardsley-EMspectrumofF1}.
\end{example}
			
		\end{appendices}

		\printbibliography
		
		%\bibliographystyle{alpha}
		%\bibliography{references}

	\end{document}